\documentclass[a4paper,11pt,reqno]{amsart}
\usepackage[left=26mm,right=26mm,top=30mm,bottom=30mm]{geometry}

\usepackage{amssymb}
\usepackage[shortlabels]{enumitem}
\usepackage{bm}
\usepackage{mathtools}
\usepackage{mathrsfs}
\usepackage{array}    
\usepackage{xparse}
\usepackage[sort&compress,numbers]{natbib}
\usepackage{caption}
\usepackage{xcolor}
\usepackage{lineno}

\newtheorem{theorem}{Theorem}[section]
\theoremstyle{plain}
\newtheorem{definition}{Definition}[section]
\newtheorem{lemma}{Lemma}[section]
\newtheorem{proposition}{Proposition}[section]
\newtheorem{remark}{Remark}[section]
\numberwithin{equation}{section}
\def\<#1>{\mathinner{\langle#1\rangle}}

\makeatletter
\newcommand\makebig[2]{%
  \@xp\newcommand\@xp*\csname#1\endcsname{\bBigg@{#2}}%
  \@xp\newcommand\@xp*\csname#1l\endcsname{\@xp\mathopen\csname#1\endcsname}%
  \@xp\newcommand\@xp*\csname#1r\endcsname{\@xp\mathclose\csname#1\endcsname}%
}
\makeatother

\makebig{biggg} {3.0}
\makebig{Biggg} {3.5}
\makebig{bigggg}{4.0}
\makebig{Bigggg}{4.5}

\begin{document}
\title{On the exponential stability of Beck's Problem on a star-shaped graph}
\author{Mahyar Mahinzaeim$^{\tt1,\ast}$}
\author{Gen Qi Xu$^{\tt2}$}
\author{Hai E Zhang$^{\tt2}$}

 \thanks{
\vspace{-1em}\newline\noindent
\newline\noindent
$^{\tt1}$ Research Center for Complex Systems, Aalen University, Germany.
  \newline\noindent
 $^{\tt2}$ Department of Mathematics, Tianjin University, China.
    \newline\noindent
  {\sc Emails}:
 {\tt m.mahinzaeim@web.de},~{\tt gqxu@tju.edu.cn},~{\tt ghaiezhang@126.com}.
    \newline\noindent
$^{\ast}$ Corresponding author.
}

\begin{abstract}
 We deal with the as yet unresolved exponential stability problem for Beck's Problem on a metric star graph with three identical edges. The edges are  stretched Euler--Bernoulli beams which are simply supported with respect to the outer vertices. At the inner vertex we have viscoelastic damping acting on the slopes of the edges. We carry out a complete spectral analysis of the system operator associated with the abstract spectral problem in Hilbert space. Within this framework it is shown that the eigenvectors have the property of forming a Riesz (i.e.\ an unconditional) basis, which makes it possible to directly deduce the exponential stability of the corresponding $C_0$-semigroup using spectral information for the system operator alone. A physically interesting conclusion is that the particular choice of vertex conditions ensures the exponential stability even when the elasticity acting on the slopes of the edges is absent.
\end{abstract}
\maketitle

\pagestyle{myheadings} \thispagestyle{plain} \markboth{\sc M.\ Mahinzaeim, G.\ Q.\ Xu, and H.\ E.\ Zhang}{\sc Beck's Problem on a star graph}

\section{Introduction}\label{sec_intro}

The small motions of a long, thin, uniform Euler--Bernoulli beam of finite length which is subjected to an external axial force negatively proportional to the bending moment can be modelled by the partial differential equation
\begin{equation}\label{eq_01}
\frac{\partial^2}{\partial t^2}\bm{w}\left(s,t\right)-\gamma\frac{\partial^2}{\partial s^2}\bm{w}\left(s,t\right)+\frac{\partial^4}{\partial s^4}\bm{w}\left(s,t\right)=0.
\end{equation}
Here $\bm{w}\left(s,t\right)$ is the deflection of the beam at position $s$ and time $t$. In the case where the beam is compressed one has $\gamma<0$, otherwise, when $\gamma>0$, the beam is stretched. By supplementing \eqref{eq_01} with the standard boundary conditions of clamped-free ends, one gets the physical problem known as \textit{Beck's Problem}, which is one of the oldest problems in the theory of elastic stability (see \cite{Beck1952,Bolotin1963,Ziegler1977}). An analysis of the existence, location, multiplicity and asymptotics of the eigenvalues, i.e., a complete spectral analysis of the boundary-eigenvalue problem associated with \eqref{eq_01}, subject to damped, simply supported (pinned) boundary conditions has been given in \cite{MollerPivovarchik2006} in the spatially variable compression/tension case.

Our primary interest in this paper lies with the dynamic behaviour and stability of a star-shaped network of such beams, with a dampening effect at the junction point of the network. This is an interesting problem which should find application in a varied assortment of mechanical and control problems on networks. Beck's Problem, of course, is nothing new by itself, but the problem we consider is an interpretation as a network setup of Beck's Problem with damped simple-elastic ends. Even in this comparatively simple case (the star-shaped network is an elementary example of a nontrivial network) we are not aware that the stability problem, let alone the exponential stability problem has been addressed previously. From the physical viewpoint the natural question the engineer will therefore ask (but which so far has escaped attention in the mathematical literature) is this: \textit{Do the connectivity conditions coupling the beams cause all motions of the network to die out, possibly exponentially, as time progresses?}

To explain our problem in a bit more detail, let us consider first \eqref{eq_01} on each beam or what we call edge $e$:
\begin{equation*}
\frac{\partial^2}{\partial t^2}\bm{w}_e\left(s,t\right)-\gamma\frac{\partial^2}{\partial s^2}\bm{w}_e\left(s,t\right)+\frac{\partial^4}{\partial s^4}\bm{w}_e\left(s,t\right)=0,\quad s\in e.
\end{equation*}
The existence and properties of solutions $\bm{w}_e$ also depend on initial conditions, boundary conditions at the outer vertices of the network, and certain connectivity conditions at the inner vertex. We are going to consider an equilateral $3$-edge metric star graph, i.e.\ a metric star graph formed by three identical edges $e_1$, $e_2$, $e_3$, as illustrated in Fig.\ \ref{fig01}, where each edge is of the same finite length and there is a single inner vertex at $a_0$. The boundary conditions at the outer vertices $a_1$, $a_2$, $a_3$ correspond to the edges being simply supported at one end. At the inner vertex we have a combination of elasticity and viscous damping, or \textit{viscoelastic damping} acting on the angles or slopes of the edges. It turns out that the question of stability essentially concerns (not so surprisingly) the viscous damping condition at the inner vertex but (surprisingly) not the elasticity condition, as long as $\gamma> 0$.

\begin{figure}[!h]
\begin{center}
\includegraphics[width=0.45\textwidth]{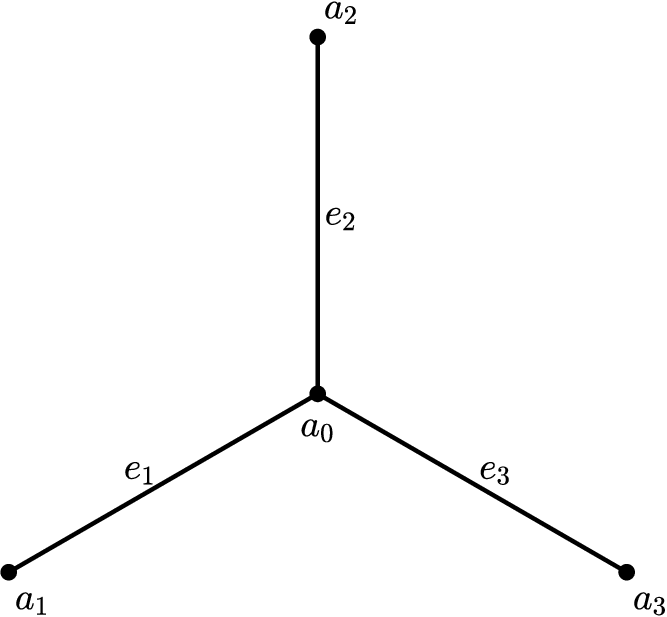}
\end{center}
\caption{An equilateral $3$-edge metric star graph.}\label{fig01}
\end{figure}

The problem with the system considered, and in fact with all continuum mechanical or distributed systems governed by partial differential equations, is that the spectral mapping theorem in general need not hold. This difficulty prevents use of only spectral information for the \textit{system operator} (or, equivalently, for the corresponding spectral problem) to study the exponential stability of the corresponding {strongly continuous} semigroup for an abstract or operator formalism in \textit{infinite-dimensional} {state space}. By exponential stability of the semigroup we mean that the norms of solutions of the corresponding first-order initial-value problem decay exponentially to zero as time goes to infinity; a formal definition will be given later. So it is clear from what has just been said that the mathematical question directly related to the question in the second paragraph then arises: \textit{Does the spectrum of the system operator determine the decay rate of the semigroup?}

As indicated above, the key to our approach to the stability problem is the abstract formulation in terms of operators. Let ${S}\left(t\right)$ be a strongly continuous semigroup of bounded linear operators, in short $C_0$-semigroup, on the Hilbert space $\mathbb{X}$ with infinitesimal generator $\mathcal{T}$ (a closed linear operator with domain $\bm{D}\left(\mathcal{T}\right)$ dense in $\mathbb{X}$). The operator $\mathcal{T}$ is the system operator in our contexts. In general, then, one always has by the Hille--Yoshida theorem that
\begin{equation}\label{eqwwr456}
\omega\left(\mathcal{T}\right)\geq s\left(\mathcal{T}\right),
\end{equation}
where $\omega\left(\mathcal{T}\right)\coloneqq\lim_{t\rightarrow\infty}\frac{1}{t}\log\left\|{S}\left(t\right)\right\|$ is the type, $s\left(\mathcal{T}\right)\coloneqq \sup\left\{\operatorname{Re}{\lambda}~\middle|~\lambda\in\sigma\left(\mathcal{T}\right)\right\}$ is the spectral abscissa, and $\sigma\left(\mathcal{T}\right)$ denotes the spectrum of $\mathcal{T}$. If equality holds in \eqref{eqwwr456}, then, as is well known, $\mathcal{T}$ satisfies the \textit{spectrum-determined growth assumption}. If
\begin{equation}\label{eqwwr457}
\sup\left\{\operatorname{Re}{\lambda}~\middle|~\lambda\in\sigma\left(\mathcal{T}\right)\right\}\leq -\varepsilon<0
\end{equation}
for some constant $\varepsilon> 0$, then the exponential stability of $S\left(t\right)$ is a
direct consequence of the spectrum-determined growth assumption. Hence, in this case, $\varepsilon$ is the decay rate and is intimately related to the spectral abscissa. Let us note that even if we can show that $\mathcal{T}$ satisfies the spectrum-determined growth assumption, we must still face the problem of analysing and hence estimating its spectrum, which is a good deal more complicated given the network nature of the system. This means, effectively, that the object of the paper is twofold: prove that we have equality in \eqref{eqwwr456} and, in particular, provide accurate estimates of the location and asymptotics of the spectrum of $\mathcal{T}$ thereby show that condition \eqref{eqwwr457} holds. This is the main undertaking of the paper and it is hoped that the methods we employ, and the results we obtain, will serve as a guide for general star graphs of beams. (The reader will undoubtedly notice that our results may in fact be extended, without additional effort, to graphs of more general topology. The symmetry of the graph about the vertices does probably play a role, however, when it comes to spectral analyses.)

Our results are naturally obtained in the context of the classical spectral problem
\begin{equation}\label{eq_01xaa02}
\mathcal{T}x=\lambda x,\quad x\in\bm{D}\left(\mathcal{T}\right)\subset \mathbb{X},\quad \lambda\in\mathbf{C},
\end{equation}
for the root vectors $x$ (eigen- and associated vectors or chains of eigen- and associated vectors). Here $\lambda$ is the spectral parameter, and it is important to note that the system operator $\mathcal{T}$ is not skewadjoint in general. We seek to prove the spectrum-determined growth assumption by proving that $\mathcal{T}$ is a discrete operator (employing some old terminology of Dunford, see \cite[Definition XIX.1]{DunfordSchwartz1971}) -- i.e.\ it has a compact resolvent and thus a discrete spectrum consisting only of normal eigenvalues (isolated eigenvalues of finite algebraic multiplicity, see Definition \ref{d05xa}) and its root vectors form a \textit{Riesz basis} for the underlying Hilbert space $\mathbb{X}$. A bonus of using this ``spectral approach'' is that the Riesz basis property or Riesz basisness of the root vectors also gives the semigroup generating property and guarantees the representation of solutions of the initial-value problem as norm-convergent series of eigensolutions. All these considerations already show that in effect we reduce the study of exponential stability to using only spectral information for the system operator, as one would do in lumped parameter systems modelled by ordinary differential equations (cf., e.g., \cite{Bolotin1963}).

The paper is organised as follows. In Section \ref{sec_3} we lay out some preliminary material for use in the paper, pose the initial/boundary-value problem, and formulate the abstract first-order initial-value problem associated with it. Function spaces and operators basic to the abstract formulation are defined in the section. We will use some terminology common in the literature of differential equations on networks or on metric graphs, and we recommend the reader consult the papers \cite{MR2070600,MR2070601} for definition of the terms and development of the related theory (from a somewhat historical perspective). We also refer the interested reader to the text \cite{Mugnolo2014} where a good treatment of the semigroup formulation of such systems can be found. Some stability (and stability-related or control) problems for such systems in the context of string, beam, and plate networks are well covered in the monographs \cite{DagerZuazua2006,LagneseEtAl1994,MR4397494,Xu2010graph}. We should also note the very readable monograph \cite{MollerPivovarchik2015} in which the authors present a variety of special examples from mechanics of metric graphs. In Section \ref{sec_3a} we examine the well-posedness of the initial-value problem, which requires study of the semigroup generating property of the system operator $\mathcal{T}$. The heart of this work -- also from a technical standpoint -- lies in Sections \ref{sec_3ax} and \ref{sec_6}, where we establish the key results required for the proof of exponential stability in Section \ref{sec_7}. While Section \ref{sec_3ax} is devoted to a complete spectral analysis of $\mathcal{T}$, in Section \ref{sec_6} the completeness, minimality, and Riesz basisness of the corresponding root vectors (all of which shown in Theorem \ref{T-5-3} to be eigenvectors) are investigated with the aid especially of two results due respectively to Keldysh \cite[Theorem V.8.1]{GohbergKrein1969} and Xu and Yung \cite[Theorem 1.1]{xu2005expansion}. The former is a well-known result on the completeness of root vectors of compact operators on a Hilbert space which are subject to certain nuclearity conditions. The latter is on the Riesz basisness (with parentheses) and is particularly well suited for our purposes because it exploits the semigroup description of our physical model; the result may, in fact, be compared with comparable results in \cite{MR2718703,MR2407202,MR3340175}. We refer the reader to the standard texts \cite{GohbergKrein1969,Markus1988,Lyubich1992} for definitions of completeness, minimality, unconditional or Riesz bases, and so on (and further details concerning various tests for these).

\section{Problem formulation and preliminaries}\label{sec_3}

As mentioned in the Introduction, our results will be obtained in the framework of abstract function spaces. We begin, however, with a detailed description of the physical model.

\subsection{The physical model}
We define the graph $\mathscr{G}\coloneqq\left(\mathscr{V},\mathscr{E}\right)$ with
\begin{equation*}
\mathscr{V}=\left\{a_0\right\}\cup\left\{ a_j\right\}_{j=1}^3,\quad \mathscr{E}=\left\{e_j\right\}_{j=1}^3,
\end{equation*}
the vertex and edge sets, respectively. Each of the edges $e_j$ connecting the inner vertex $a_0$ to the outer vertices $a_j$ is of unit length and is identified with the interval $s_j\in\left[0,1\right]$, $j=1,2,3$. The values  $s_j =0$ and $s_j =1$ correspond to the inner and outer vertices, respectively. The graph $\mathscr{G}$ thus defined is a metric graph.

Letting $\bm{w}_j\left(s_j,t\right)$ be the deflections of the edges $e_j$ for $s_j\in\left[0,1\right]$ at time $t\ge 0$ from their equilibrium positions which we identify with $\mathcal{G}$, we suppose they satisfy the partial differential equation
\begin{equation}\label{eq_02}
\frac{\partial^2}{\partial t^2}\bm{w}_j\left(s_j,t\right)-\gamma \frac{\partial^2}{\partial s_j ^2}\bm{w}_j\left(s_j,t\right)+\frac{\partial^4}{\partial s_j ^4}\bm{w}_j\left(s_j,t\right)=0,\quad j=1,2,3.
\end{equation}
Associated with \eqref{eq_02} we have given initial conditions
\begin{equation}\label{eqnew21}
\bm{w}_j\left(s_j,0\right)=g_j\left(s_j\right),\quad \left.\frac{\partial}{\partial t}\bm{w}_j\left(s_j,t\right)\right|_{t=0}=h_j\left(s_j\right),\quad s_j\in e_j,\quad j=1,2,3,
\end{equation}
where the known functions $g_j$, $h_j$ are ``suitably'' smooth as specified later. The edges are assumed simply supported at the outer vertices $a_j$, so the boundary conditions at $s_j =1$ are
\begin{equation}\label{eq_03}
\bm{w}_j\left(1,t\right)=\left.\frac{\partial^2}{\partial s_j ^2}\bm{w}_j\left(s_j,t\right)\right|_{s_j =1}=0,\quad j=1,2,3.
\end{equation}
At the inner vertex various connectivity conditions may be considered; we impose the following:
\begin{itemize}[leftmargin=*,align=left,labelwidth=\parindent]
\item\label{item01} {Deflection compatibility condition:}
\begin{equation}\label{eq_04}
\bm{w}_i\left(0,t\right)=\bm{w}_j\left(0,t\right),\quad i,j=1,2,3.
\end{equation}
\item\label{item02} {Viscoelastic damping condition:}
\begin{equation}\label{eq_05}
\left.\left(\frac{\partial^2}{\partial s_j ^2}\bm{w}_j\left(s_j,t\right)-\alpha\frac{\partial}{\partial s_j }\bm{w}_j\left(s_j,t\right)-\beta\frac{\partial^2}{\partial s_j \partial t}\bm{w}_j\left(s_j,t\right)\right)\right|_{s_j =0}=0,\quad j=1,2,3.
\end{equation}
\item\label{item03} {Force balance condition:}
\begin{equation}\label{eq_06}
\sum^3_{j=1}\left.\left(\frac{\partial^3}{\partial s_j ^3}\bm{w}_j\left(s_j,t\right)-\gamma \frac{\partial}{\partial s_j }\bm{w}_j\left(s_j,t\right)\right)\right|_{s_j =0}=0.
\end{equation}
\end{itemize}

These conditions seem natural. Condition \eqref{eq_04} ensures the continuity or compatibility of the deflections of the edges across the inner vertex, for otherwise the edges would not be joined to the network. Condition \eqref{eq_05} recognises that the bending moments at the inner vertex are proportional to linear combinations of the slopes and the angular velocities at the inner vertex. It corresponds, more concretely, to the slopes at the vertex being acted on by elastic and viscous damping effects which are proportional, respectively, to $\alpha$ and $\beta$. Condition \eqref{eq_06} is important in that it balances the internal forces acting on the inner vertex, assuming that the edges undergo ``nonfollowing'' tension or compression proportional to $\gamma$. By this is meant that the external axial forces remain constant in direction along the axes of the edges, projected in the direction of the shear forces. It is crucial to understand that this assumption, although physically reasonable, makes the originally nonconservative Beck's Problem conservative in the absence of dissipative effects (i.e.\ when $\beta=0$). Strictly speaking then, Beck's Problem specialises to the classical problem of Euler, but we shall not stress this distinction between the two problems in the paper (see the books \cite{Bolotin1963,Ziegler1977} and the survey \cite{Elishakoff2005} for further discussion of this point). Throughout the paper we will make the assumption that the edges are stretched for fixed $\gamma >0$ and let the parameters $\alpha$, $\beta$ vary. (We expect our techniques to apply similarly to the compression modification, i.e.\ when $\gamma<0$.)

The problem posed by \eqref{eq_02}--\eqref{eq_06} constitutes the initial/boundary-value problem which, we should stress, is considerably more complicated and does not fit precisely into the framework considered, e.g., in \cite{Xu2012} or \cite[Section 4.2]{Xu2010graph} (as well as in the references contained therein). So in what follows we suggest an alternative approach to abstracting the initial/boundary-value problem and to verification of the necessary theoretical concepts.

\subsection{Operator formulation}
Let us begin by putting $s$ in place of the variable $s_j\in\left[0,1\right]$, and set
\begin{equation*}
\bm{v}_j\left(s,t\right)=\frac{\partial}{\partial t}\bm{w}_j\left(s,t\right),\quad \bm{x}_j\left(s,t\right)=\left(\begin{matrix}
\bm{w}_j\left(s,t\right)\\[0.2em]
\bm{v}_j\left(s,t\right)
\end{matrix}\right),\quad j=1,2,3,
\end{equation*}
and
\begin{equation*}
\bm{x}\left(s,t\right)=\left(\begin{matrix}
\bm{x}_1\left(s,t\right)\\[0.2em]
\bm{x}_2\left(s,t\right)\\[0.2em]
\bm{x}_3\left(s,t\right)
\end{matrix}\right).
\end{equation*}
Denoting by $\bm{H}^m\left(0,1\right)$, $m\in\mathbf{N}_0\left(=\mathbf{N}\cup\left\{0\right\}\right)$, the Sobolev space with $\bm{L}_2\left(0,1\right)\coloneqq\bm{H}^0\left(0,1\right)$, we define the space
\begin{equation*}
\hat{\bm{H}}^2\left(0,1\right)\coloneqq\left\{w\in \bm{H}^2\left(0,1\right)~\middle|~w\left(1\right)=0\right\}
\end{equation*}
which is a closed subspace of $\bm{H}^2\left(0,1\right)$. With the inner product
\begin{equation*}
\left(w,\tilde{w}\right)=\int^1_0w''\left(s\right)\overline{\tilde{w}''\left(s\right)}\,ds+\alpha w'\left(0\right)\overline{\tilde{w}'\left(0\right)}
\end{equation*}
the space $\hat{\bm{H}}^2\left(0,1\right)$ is a Hilbert space for $\alpha>0$. To proceed further it is necessary to define the metric spaces of three-component vectors in the following way:
\begin{equation*}
\bm{L}_2\left(\mathscr{G}\right)\coloneqq\left\{v=\left(\begin{matrix}
v_1\\[0.2em]
v_2\\[0.2em]
v_3
\end{matrix}\right)~\middle|~ v_j\in \bm{L}_2\left(0,1\right),\quad j=1,2,3\right\}
\end{equation*}
and
\begin{equation*}
\hat{\bm{H}}^2\left(\mathscr{G}\right)\coloneqq\left\{
w=\left(\begin{matrix}
w_1\\[0.2em]
w_2\\[0.2em]
w_3
\end{matrix}\right)~\middle|~\begin{gathered}
w_j\in \hat{\bm{H}}^2\left(0,1\right),\\[0.2em]
w_i\left(0\right)=w_j\left(0\right),\quad i,j=1,2,3
\end{gathered}
\right\}.
\end{equation*}
With these we define by
\begin{equation}\label{sspace45hgt}
\mathbb{X}= \hat{\bm{H}}^2\left(\mathscr{G}\right)\times \bm{L}_2\left(\mathscr{G}\right)\coloneqq\left\{
x=\left\{x_j\right\}^3_{j=1}~\middle|~\begin{gathered}
x_j=\left(\begin{matrix}
w_j\\[0.2em]
v_j
\end{matrix}\right)\in \hat{\bm{H}}^2\left(0,1\right)\times \bm{L}_2\left(0,1\right),\\[0.2em]
w_i\left(0\right)=w_j\left(0\right),\quad i,j=1,2,3
\end{gathered}
\right\}
\end{equation}
the state space with the inner product
\begin{equation*}
\<x,\tilde{x}> \coloneqq\left(w,\tilde{w}\right)_{2}+\left(w,\tilde{w}\right)_1+\left(v,\tilde{v}\right)_0,
\end{equation*}
where
\begin{gather*}
\left(w,\tilde{w}\right)_{2}=\sum_{j=1}^3\left(\int^1_0w_j''\left(s\right)\overline{\tilde{w}_j''\left(s\right)}\,ds+\alpha w'_j\left(0\right)\overline{\tilde{w}'_j\left(0\right)}\right),\quad
\left(w,\tilde{w}\right)_1=\gamma\sum_{j=1}^3 \int^1_0w_j'\left(s\right)\overline{\tilde{w}_j'\left(s\right)}\,ds,\\[0.2em]
\left(v,\tilde{v}\right)_0=\sum_{j=1}^3\int^1_0v_j\left(s\right)\overline{\tilde{v}_j\left(s\right)}\,ds.
\end{gather*}
The resulting norm in $\mathbb{X}$ will be denoted by $\left\|\,\cdot\,\right\| $, as usual.

In $\mathbb{X}$ we define the system operator $\mathcal{T}$ on the domain
\begin{equation}\label{eq_08}
\bm{D}\left(\mathcal{T}\right)=\left\{x=\left\{x_j\right\}^3_{j=1}\in \mathbb{X}~\middle|
~\begin{gathered}
x_j=\left(\begin{matrix}
w_j\\[0.2em]
v_j
\end{matrix}\right)\in (\bm{H}^4\left(0,1\right)\cap \hat{\bm{H}}^2\left(0,1\right))\times \hat{\bm{H}}^2\left(0,1\right), \\[0.2em]
w''_j\left(1\right)=0,\quad w''_j\left(0\right)-\alpha w'_j\left(0\right)-\beta v'_j\left(0\right)=0,\quad j=1,2,3,\\
\sum_{j=1}^3\,(w^{(3)}_j\left(0\right)-\gamma  w'_j\left(0\right))=0
\end{gathered}\right\}
\end{equation}
by 
\begin{equation}\label{eq_07}
\mathcal{T}{x}\coloneqq \left\{\left(\begin{matrix}
v_j\\[0.2em]
 -w_j^{(4)}+\gamma w''_j
\end{matrix}\right)
\right\}_{j=1}^3.
\end{equation}
Then we can rewrite the initial/boundary-value problem described by \eqref{eq_02}--\eqref{eq_06} as an abstract first-order initial-value problem in $\mathbb{X}$ in the form
\begin{equation}\label{eq_09}
\left\{\begin{gathered}
\dot{x}\left(t\right)=\mathcal{T}{x}\left(t\right),\quad {x}\left(t\right)=\bm{x}\left(\,\cdot\,,t\right)=\left\{\left(\begin{matrix}
\bm{w}_j\left(\,\cdot\,,t\right)\\[0.2em]
\bm{v}_j\left(\,\cdot\,,t\right)
\end{matrix}\right)\right\}^3_{j=1},\\ {x}\left(0\right)=x_0=\left\{\left(\begin{matrix}
g_j\\[0.2em]
h_j
\end{matrix}\right)\right\}^3_{j=1}.
\end{gathered}\right.
\end{equation}
The initial state $x_0$ is assumed to be suitably smooth in the sense that $x_0\in\bm{D}\left(\mathcal{T}\right)$. We recall two basic notions from semigroup theory:
\begin{itemize}[leftmargin=*,align=left,labelwidth=\parindent]
\item Let ${x}\left(t\right)={S}\left(t\right)x_0$, continuously differentiable and satisfying \eqref{eq_09} in $\mathbb{X}$ for all $t\in\mathbf{R}_+$, where ${S}\left(t\right)$ is a $C_0$-semigroup on $\mathbb{X}$ with infinitesimal generator $\mathcal{T}$. Then ${x}\left(t\right)$ is a \textit{classical} solution, or simply a solution of \eqref{eq_09}.
\item The $C_0$-semigroup generated by $\mathcal{T}$ is said to be \textit{exponentially} stable if and only if there are constants $M,\varepsilon>0$ such that
\begin{equation}\label{eqnewa1}
\left\|{S}\left(t\right)\right\| \leq M e^{-\varepsilon t},\quad t\ge 0.
\end{equation}
\end{itemize}

The goal in the paper now clearly becomes one of proving that for solutions of \eqref{eq_09} we have \eqref{eqnewa1}. For this, according to our programme stated in the Introduction, we must check whether equality in \eqref{eqwwr456} holds and condition \eqref{eqwwr457} is satisfied. We must first however check to see if \eqref{eq_09} is well posed.

\section{Well-posedness}\label{sec_3a}

The next two lemmas deal with properties of $\mathcal{T}$ defined by \eqref{eq_08}, \eqref{eq_07} and will be also of importance later in the proofs of the main theorems. We recall that, by definition,
\begin{equation*}
\alpha,\beta\geq 0,\quad \gamma>0
\end{equation*}
throughout the whole paper.
\begin{lemma}\label{L-3-1}
$\mathcal{T}$ has a compact inverse.
\end{lemma}
\begin{proof}
We consider the equation
\begin{equation}\label{eq12xa34}
\mathcal{T}x=\tilde{x}
\end{equation}
with $\tilde{x}\in \mathbb{X}$, $x\in \bm{D}\left(\mathcal{T}\right)$, which is equivalent to
\begin{equation}\label{eq12ssaaaxa34}
\left\{\begin{aligned}
v_j&=\tilde{w}_j,& j&=1,2,3,\\[0.2em]
-w^{(4)}_j+\gamma w''_j&=\tilde{v}_j,& j&=1,2,3,\\[0.2em]
w_j\left(1\right)=w''_j\left(1\right)&=0,& j&=1,2,3,\\[0.2em]
w_i\left(0\right)&=w_j\left(0\right),& i,j&=1,2,3,\\[0.2em]
w''_j\left(0\right)-\alpha w'_j\left(0\right)-\beta v'_j\left(0\right)&=0,& j&=1,2,3,\\[0.2em]
\sum^3_{j=1}\,(w^{(3)}_j\left(0\right)-\gamma  w'_j\left(0\right))&=0.&&&
\end{aligned}\right.
\end{equation}
Integrating the differential equations in \eqref{eq12ssaaaxa34} twice from $0$ to $1$, making use of the boundary conditions $w_j\left(1\right)=w''_j\left(1\right)=0$, $j=1,2,3$, we get
\begin{equation}\label{eq233xx}
w''_j\left(s\right)-\gamma w_j\left(s\right)+(w^{(3)}_j\left(0\right)-\gamma w_j'\left(0\right))\left(1-s\right)=-\tilde{V}_j\left(s\right),\quad j=1,2,3,
\end{equation}
with the integral terms
\begin{equation*}
\tilde{V}_j\left(s\right)=-\int^1_sdt\int^t_0\tilde{v}_j\left(r\right)dr.
\end{equation*}
The solutions of \eqref{eq233xx} can be directly computed to give
\begin{equation}\label{eq233}
\begin{split}
w_j\left(s\right)&=a_j\sinh\sqrt{\gamma}\left(1-s\right)+(w^{(3)}_j\left(0\right)-\gamma w_j'\left(0\right))\,\frac{1}{\sqrt{\gamma}}\int^1_s\left(1-r\right)\sinh\sqrt{\gamma}\left(s-r\right)dr\\
&\qquad +\frac{1}{\sqrt{\gamma}}\int^1_s\sinh\sqrt{\gamma}\left(s-r\right)\tilde{V}_j\left(r\right)dr,\quad j=1,2,3,
\end{split}
\end{equation}
where $a_j$ are arbitrary constants, and we note that
\begin{equation}\label{eq233xxy}
\begin{split}
w'_j\left(s\right)&=-a_j\sqrt{\gamma}\cosh\sqrt{\gamma}\left(1-s\right)+(w^{(3)}_j\left(0\right)-\gamma w_j'\left(0\right))\int^1_s\left(1-r\right)\cosh\sqrt{\gamma}\left(s-r\right)dr\\
&\qquad +\int^1_s\cosh\sqrt{\gamma}\left(s-r\right)\tilde{V}_j\left(r\right)dr,\quad j=1,2,3.
\end{split}
\end{equation}
Using the vertex condition $\sum^3_{j=1}\,(w^{(3)}_j\left(0\right)-\gamma  w'_j\left(0\right))=0$ in \eqref{eq233} we obtain
\begin{equation}\label{eq13a}
\sum^3_{j=1}w_j\left(0\right)=\sinh\sqrt{\gamma}\sum^3_{j=1}a_j-\frac{1}{\sqrt{\gamma}}\int^1_0\sinh\sqrt{\gamma}\,r\sum^3_{j=1}\tilde{V}_j\left(r\right)dr.
\end{equation}
Substituting \eqref{eq233} in \eqref{eq233xx} and using the result together with \eqref{eq233xxy} in the vertex condition $w''_j\left(0\right)-\alpha w'_j\left(0\right)-\beta v'_j\left(0\right)=0$, $j=1,2,3$, we have that
\begin{equation}
\begin{split}
&a_j\left(\gamma\sinh\sqrt{\gamma}+\alpha\sqrt{\gamma}\cosh\sqrt{\gamma}\right)\\
&\hspace{0.1\linewidth}-(w^{(3)}_j\left(0\right)-\gamma  w'_j\left(0\right))\left[\int^1_0\left(1-r\right)\left(\sqrt{\gamma}\sinh\sqrt{\gamma}\,r+\alpha\cosh\sqrt{\gamma}\,r\right)dr +1\right]\\
&\hspace{0.1\linewidth}-\tilde{V}_j\left(0\right)-\beta \tilde{w}'_j\left(0\right)-\int^1_0\left(\sqrt{\gamma}\sinh\sqrt{\gamma}\,r+\alpha\cosh\sqrt{\gamma}\,r\right)\tilde{V}_j\left(r\right)dr=0,\quad j=1,2,3,\label{eq1aa3a}
\end{split}
\end{equation}
where we have taken into account that $v_j=\widetilde{w}_j$, $j=1,2,3$. Summing \eqref{eq1aa3a} over $j=1,2,3$ we find that
\begin{equation}\label{eq12a}
\sum^3_{j=1} a_j=\frac{\sum^3_{j=1}\left[\tilde{V}_j\left(0\right)+\beta \tilde{w}'_j\left(0\right)+\int^1_0\left(\sqrt{\gamma}\sinh\sqrt{\gamma}\,r+\alpha\cosh\sqrt{\gamma}\,r\right)\tilde{V}_j\left(r\right)dr\right]}{\gamma\sinh\sqrt{\gamma}+\alpha\sqrt{\gamma}\cosh\sqrt{\gamma}}\coloneqq b\left(\tilde{w},\tilde{v}\right).
\end{equation}
Then from the compatibility condition $w_i\left(0\right)=w_j\left(0\right)\equiv w\left(0\right)$, $i,j=1,2,3$, and using \eqref{eq12a} in \eqref{eq13a} we have
\begin{equation}\label{eqne03}
w\left(0\right)=\frac{1}{3}\left(b\left(\tilde{w},\tilde{v}\right)\sinh\sqrt{\gamma}-\frac{1}{\sqrt{\gamma}}\int^1_0\sinh\sqrt{\gamma}\,r\sum^3_{j=1}\tilde{V}_j\left(r\right)dr\right)\coloneqq c\left(\tilde{w},\tilde{v}\right).
\end{equation}
Combining \eqref{eq233} and \eqref{eq1aa3a}, taking into account \eqref{eqne03}, we arrive at the algebraic equations
\begin{equation}\label{eq_12a}
\left\{\begin{split}
&a_j\sqrt{\gamma}\sinh\sqrt{\gamma}-(w^{(3)}_j\left(0\right)-\gamma  w'_j\left(0\right))\int^1_0\left(1-r\right)\sinh\sqrt{\gamma}\,r\,dr\\[0.2em]
&\hspace{0.4\linewidth}=\sqrt{\gamma}\,c\left(\tilde{w},\tilde{v}\right)
+\int^1_0\sinh\sqrt{\gamma}\,r\tilde{V}_j\left(r\right)dr,\\[0.2em]
&a_j\left(\gamma\sinh\sqrt{\gamma}+\alpha\sqrt{\gamma}\cosh\sqrt{\gamma}\right)\\
&\hspace{0.1\linewidth}-(w^{(3)}_j\left(0\right)-\gamma  w'_j\left(0\right))\left[\int^1_0\left(1-r\right)\left(\sqrt{\gamma}\sinh\sqrt{\gamma}\,r+\alpha\cosh\sqrt{\gamma}\, r\right)dr +1\right]\\[0.2em]
&\hspace{0.4\linewidth}=\int^1_0\left(\sqrt{\gamma}\sinh\sqrt{\gamma}\, r+\alpha\cosh\sqrt{\gamma}\,r\right)\tilde{V}_j\left(r\right)dr\\
&\hspace{0.4\linewidth}\qquad+\tilde{V}_j\left(0\right)+\beta \tilde{w}'_j\left(0\right)
\end{split}\right.
\end{equation}
for $a_j$, $w^{(3)}_j\left(0\right)-\gamma w'_j\left(0\right)$, $j=1,2,3$. Since the determinant
\begin{equation*}
\left|\begin{matrix}
 \sinh\sqrt{\gamma}& -\int^1_0\left(1-r\right)\sinh\sqrt{\gamma}\,r\,dr\\[0.2em]
\alpha\cosh\sqrt{\gamma}& -\alpha\int^1_0\left(1-r\right)\cosh\sqrt{\gamma}\,r\,dr -1
\end{matrix}\right|\neq 0,
\end{equation*}
\eqref{eq_12a} has unique solution pairs $a_j$, $w^{(3)}_j\left(0\right)-\gamma w'_j\left(0\right)$. Therefore \eqref{eq12xa34} has the unique solution
\begin{equation}\label{eq12xvbn}
\left\{\left(\begin{matrix}
w_j\\[0.2em]
v_j
\end{matrix}\right)\right\}^3_{j=1}=\mathcal{T}^{-1}\left\{\left(\begin{matrix}
\tilde{w}_j\\[0.2em]
\tilde{v}_j
\end{matrix}\right)\right\}^3_{j=1}\in \bm{D}\left(\mathcal{T}\right)
\end{equation}
which proves, by the closed graph theorem, that the inverse $\mathcal{T}^{-1}$ of $\mathcal{T}$ is closed and bounded. Since $\bm{D}\left(\mathcal{T}\right)\subset (\bm{H}^4\left(\mathscr{G}\right)\cap \hat{\bm{H}}^2\left(\mathscr{G}\right))\times \hat{\bm{H}}^2\left(\mathscr{G}\right)\subset \mathbb{X}$, we conclude from Sobolev's embedding theorem that $\mathcal{T}^{-1}$ is a compact operator on $\mathbb{X}$.
\end{proof}

\begin{remark}\label{rem01245}
The assumption $\gamma>0$ cannot be weakened to $\gamma\geq0$ in the proof of the lemma. In particular this means -- see Definition \ref{D-3-1} and Theorem \ref{T-3-2} in the next section -- that $0$ is excluded from the spectrum of $\mathcal{T}$, irrespective of the value of $\alpha$, so long as $\gamma>0$.
\end{remark}

\begin{lemma}\label{L-3-2}
$\mathcal{T}$ is maximal dissipative when $\beta>0$ and skewadjoint when $\beta=0$.
\end{lemma}
\begin{proof}
A direct computation using integration by parts gives for any $x\in \bm{D}\left(\mathcal{T}\right)$
\begin{align}
2\operatorname{Re}\<\mathcal{T}x,x> &=\<\mathcal{T}x,x> +\<x,\mathcal{T}x> \nonumber\\
&=-2\beta\sum^3_{j=1}\left|v'_j\left(0\right)\right|^2-2\operatorname{Re}\sum^3_{j=1}\,(w^{(3)}_j\left(1\right)-\gamma  w'_j\left(1\right))\,\overline{{v}_j\left(1\right)}\nonumber\\
&=-2\beta\sum^3_{j=1}\left|v'_j\left(0\right)\right|^2,\label{eq_13}
\end{align}
with the last line a consequence of the fact that ${v}_j\left(1\right)=0$, $j=1,2,3$, by definition. Thus $\mathcal{T}$ is dissipative when $\beta>0$, and skewsymmetric when $\beta=0$. Maximality of $\mathcal{T}$ follows directly from Lemma \ref{L-3-1}. Indeed, since by Lemma \ref{L-3-1} we have $0\in\varrho\left(\mathcal{T}\right)$, the resolvent set of $\mathcal{T}$ (see Definition \ref{D-3-1}), application of the contraction fixed point theorem shows that there exists a sufficiently small $\lambda>0$ such that the range $\operatorname{Im}\left(\lambda I-\mathcal{T}\right)=\mathbb{X}$.

For the proof of skewadjointess we take $\beta=0$ in \eqref{eq_13} and obtain, with $\mathcal{T}x=y$, $y\in\mathbb{X}$, $x\in\bm{D}\left(\mathcal{T}\right)$,
\begin{equation*}
\langle y,\mathcal{T}^{-1}y\rangle=\langle \mathcal{T}\mathcal{T}^{-1}y,\mathcal{T}^{-1}y \rangle=-\langle \mathcal{T}^{-1}y,\mathcal{T}\mathcal{T}^{-1}y \rangle=-\langle \mathcal{T}^{-1}y,y \rangle.
\end{equation*}
It follows that $\mathcal{T}^{-1}$, as it is bounded, is skewadjoint for $\beta=0$; hence $\mathcal{T}$ is skewadjoint for $\beta=0$.
\end{proof}

With the fact that the closed, maximal dissipative system operator $\mathcal{T}$ is densely defined (because $\mathbb{X}$ is a Hilbert space), we have the following result immediately from the Lumer--Phillips theorem (see, e.g., \cite[Section II.3.b]{EngelNagel1999}, \cite[Section 1.4]{Pazy1983} or \cite[Section I.4.2]{Krein1971} for details).
\begin{theorem}\label{T-3-1}
$\mathcal{T}$ is the infinitesimal generator of a $C_0$-semigroup of contractions ${S}\left(t\right)$ on $\mathbb{X}$. So \eqref{eq_09} is well posed in the sense that for any $x_0\in \bm{D}\left(\mathcal{T}\right)$ it has a unique solution $x\in \bm{C}^1\left(\mathbf{R}_+; \mathbb{X}\right)\cap \bm{C}\left(\left[0,\infty\right); \bm{D}\left(\mathcal{T}\right)\right)$.
\end{theorem}

\section{Spectral analysis}\label{sec_3ax}

As stated in the Introduction, we wish to approach the stability problem by investigating the spectral problem for \eqref{eq_09}, but we now can be more specific. Performing separation of variables in \eqref{eq_09} of the form ${x}\left(t\right)=x\exp\left(\lambda t\right)$, $x\in\mathbb{X}$, $\lambda\in\mathbf{C}$, it is easily seen that the resulting spectral problem is of exactly the same form as \eqref{eq_01xaa02} with $\mathbb{X}$ and $\mathcal{T}$ being defined by \eqref{sspace45hgt} and \eqref{eq_08}, \eqref{eq_07}, respectively, and is equivalent to the boundary-eigenvalue problem
\begin{equation}\label{eq1sdcxabv}
\left\{\begin{aligned}
v_j&=\lambda w_j,& j&=1,2,3,\\[0.2em]
-w^{(4)}_j+\gamma w''_j&=\lambda v_j,& j&=1,2,3,\\[0.2em]
w_j\left(1\right)=w''_j\left(1\right)&=0,& j&=1,2,3,\\[0.2em]
w_i\left(0\right)&=w_j\left(0\right),& i,j&=1,2,3,\\[0.2em]
w''_j\left(0\right)-\alpha w_j'\left(0\right)-\beta v_j'\left(0\right)&=0,& j&=1,2,3,\\
\sum^3_{j=1}\,(w^{(3)}_j\left(0\right)-\gamma w_j'\left(0\right))&=0.&&&
\end{aligned}\right.
\end{equation}
The eigenvalues of the boundary-eigenvalue problem \eqref{eq1sdcxabv} coincide (multiplicities included) with those of $\mathcal{T}$, which together with their corresponding root vectors will be analysed in detail in the sections to follow. Then we use the following standard definitions.
\begin{definition}\label{D-3-1}
Let $\lambda\mapsto \left(\lambda I-\mathcal{A}\right)$ be a mapping from $\mathbf{C}$ into the set of closed linear operators in $\mathbb{X}$. A number $\lambda\in \mathbf{C}$ is said to belong to $\varrho\left(\mathcal{A}\right)$, the resolvent set of $\mathcal{A}$, provided $\lambda I-\mathcal{A}$ has a closed and bounded inverse; we call the inverse $\left(\lambda I-\mathcal{A}\right)^{-1}$ the resolvent of $\mathcal{A}$. If $\lambda\not\in\varrho\left(\mathcal{A}\right)$, then $\lambda$ is said to be in the spectrum $\sigma\left(\mathcal{A}\right)$ of $\mathcal{A}$. We say that a number $\lambda_0\in\mathbf{C}$ is an eigenvalue of $\mathcal{A}$ if the geometric eigenspace $\operatorname{Ker}\left(\lambda_0 I-\mathcal{A}\right)\neq\left\{0\right\}$ and there exists an eigenvector $x_0\left(\neq 0\right)$ corresponding to $\lambda_0$ such that \eqref{eq_01xaa02} is satisfied. The set of all eigenvalues of $\mathcal{A}$ forms the point spectrum of $\mathcal{A}$.
\end{definition}

\begin{definition}\label{d03}
Let ${\lambda}_0\in\mathbf{C}$ be an eigenvalue of $\mathcal{A}$. The geometric multiplicity of $\lambda_0$ is the number of linearly independent eigenvectors in a system of chains of root vectors of $\mathcal{A}$ corresponding to $\lambda_0$ and is defined as $\operatorname{dim}\operatorname{Ker}\left(\lambda_0 I-\mathcal{A}\right)$. The algebraic multiplicity of ${\lambda_0}$ is the maximum value of the sum of the lengths of chains corresponding to the linearly independent eigenvectors and is defined as $\sup_{k\in\mathbf{N}}\operatorname{dim}\operatorname{Ker}\left(\lambda_0 I-\mathcal{A}\right)^k$. We call ${\lambda_0}$ semisimple if it has equal geometric and algebraic multiplicities (i.e.\ there are no associated vectors corresponding to $\lambda_0$), and simple if it is semisimple and $\operatorname{dim}\operatorname{Ker}\left(\lambda_0 I-\mathcal{A}\right)=1$.
\end{definition}

\begin{definition}\label{d05xa}
If an eigenvalue $\lambda_0\in\mathbf{C}$ is isolated and $\lambda_0 I-\mathcal{A}$ is a Fredholm operator (see \cite[Section IV.5.1]{Kato1995} for definition), then we call $\lambda_0$ a normal eigenvalue. The set of all such eigenvalues is denoted by $\sigma_0\left(\mathcal{A}\right)$.
\end{definition}

The next result refers to the structure of the spectrum of $\mathcal{T}$ and location of eigenvalues.
\begin{theorem}\label{T-3-2}
The following statements hold:
\begin{enumerate}[\normalfont(1)]
\item\label{T-3-2-a} $\sigma\left(\mathcal{T}\right)=\sigma_0\left(\mathcal{T}\right)$.
\item\label{T-3-2-b} The number $\lambda=0$ is not an eigenvalue of $\mathcal{T}$.
\item\label{T-3-2-c} If $\lambda$ is an eigenvalue of $\mathcal{T}$, then $\overline{\lambda}$, the complex-conjugate of ${\lambda}$, is also an eigenvalue of $\mathcal{T}$. Hence $\sigma\left(\mathcal{T}\right)$ is symmetric with respect to the real axis.
\item\label{T-3-2-d} $\sigma\left(\mathcal{T}\right)$ lies in the closed left half-plane; when $\beta>0$, $\sigma\left(\mathcal{T}\right)$ is confined to the open left half-plane.
\end{enumerate}
\end{theorem}
\begin{proof}
Statements \ref{T-3-2-a} and \ref{T-3-2-b} follow from Lemma \ref{L-3-1}. In fact from the lemma it follows that $\mathcal{T}$ has a compact resolvent. So the spectrum of $\mathcal{T}$ consists of an infinite number of normal eigenvalues with no finite accumulation points (see \cite[Corollary XI.8.4]{GohbergEtAl1990}). Using that $\mathcal{T}$ is a real operator establishes statement \ref{T-3-2-c}. Indeed, conjugation of \eqref{eq_01xaa02} shows that $\overline{x}$ satisfies its conjugate spectral problem and is an eigenvector of $\mathcal{T}$ which corresponds to the eigenvalue $\overline{\lambda}$. In order to establish statement \ref{T-3-2-d}, we take the inner product of \eqref{eq_01xaa02} with the corresponding $x$,
\begin{equation*}
\<\mathcal{T}x,x> =\lambda\left\|x\right\|^2.
\end{equation*}
The real part of this equation is
\begin{equation}\label{eq_15}
\operatorname{Re}\<\mathcal{T}x ,x > =\operatorname{Re}{\lambda}\left\|x \right\| ^2,
\end{equation}
yielding by Lemma \ref{L-3-2} 
\begin{equation*}
\operatorname{Re}{\lambda}=\frac{\operatorname{Re}\<\mathcal{T}x ,x > }{\left\|x \right\| ^2}\leq 0.
\end{equation*}
This proves that the spectrum of $\mathcal{T}$ lies in the closed left half-plane. We show that if $\beta>0$, then $\operatorname{Re}{\lambda} <0$. Indeed, suppose ${\lambda}$ is a purely imaginary eigenvalue, $\operatorname{Re}{\lambda} =0$, with eigenvector $x$. Then from \eqref{eq_15} we have $\operatorname{Re}\<\mathcal{T}x ,x > =0$. The proof of Lemma \ref{L-3-2} then clearly implies, since $\beta>0$,
\begin{equation*}
\sum_{j=1}^3 \left|v'_j\left(0\right)\right|^2=0.
\end{equation*}
As
\begin{equation*}
 x=\left\{\left(\begin{matrix}
w_j\\[0.2em]
v_j
\end{matrix}\right)\right\}^3_{j=1}
\end{equation*}
is an eigenvector, so $v_j=\lambda w_j$, $j=1,2,3$, it follows that
\begin{equation*}
\left|\lambda\right|^2\sum_{j=1}^3 \left|w'_j\left(0\right)\right|^2=0,
\end{equation*}
and therefore $w'_j\left(0\right)=0$, $j=1,2,3$, since $\left|\lambda\right|>0$. In this case, it is readily verified from \eqref{eq1sdcxabv} that $w_j=w_j\left(s\right)$, $j=1,2,3$, satisfies the boundary-eigenvalue problem
\begin{equation}\label{eq_16}
\left\{\begin{aligned}
w^{(4)}_j-\gamma w''_j&=-\lambda^2w_j,& j&=1,2,3,\\[0.2em]
w_j\left(1\right)=w''_j\left(1\right)&=0,& j&=1,2,3,\\[0.2em]
w_i\left(0\right)&=w_j\left(0\right),& i,j&=1,2,3,\\[0.2em]
w'_j\left(0\right)=w''_j\left(0\right)&=0,& j&=1,2,3,\\[0.2em]
\sum^3_{j=1}w^{(3)}_j\left(0\right)&=0.&&&
\end{aligned}\right.
\end{equation}
We show that any solution of \eqref{eq_16} must be the zero solution. Notice that since $\lambda$ is a purely imaginary eigenvalue, setting $\lambda=i\mu$, $\mu\in \mathbf{R}$, we have $-\lambda^2=\mu^2>0$. Consider then the boundary-eigenvalue problem
\begin{equation}\label{eq_1s6}
\left\{\begin{aligned}
\phi^{(4)}-\gamma \phi''&=\mu^2 \phi,& &\\[0.2em]
\phi\left(1\right)=\phi''\left(1\right)&=0,& &\\[0.2em]
\phi'\left(0\right)=\phi''\left(0\right)&=0.& &
\end{aligned}\right.
\end{equation}
Assume (to reach a contradiction) that $\phi=\phi\left(s\right)$ is a nonzero solution of \eqref{eq_1s6} which, without loss of generality, we may assume to be real-valued. We multiply the differential equation in \eqref{eq_1s6} by $\phi^{(3)}$ and integrate from $0$ to $1$, performing integration by parts, and using the boundary conditions. Then
\begin{equation*}
\begin{split}
&(\phi^{(3)}\left(1\right)\!{)}^2-(\phi^{(3)}\left(0\right)\!{)}^2-\left(\int_0^1\phi^{(3)}\left(s\right)\phi^{(4)}\left(s\right)ds-\gamma\int_0^1 \phi^{(3)}\left(s\right)\phi''\left(s\right)ds\right)\\
&\hspace{0.1\linewidth}=-\mu^2\,(\phi'\left(1\right)\!{)}^2-\int_0^1 \phi^{(3)}\left(s\right)\phi\left(s\right)ds,
\end{split}
\end{equation*}
and there we have
\begin{equation*}
(\phi^{(3)}\left(1\right)\!{)}^2-(\phi^{(3)}\left(0\right)\!{)}^2=-\mu^2\,(\phi'\left(1\right)\!{)}^2
\end{equation*}
which obviously can be true if and only if $\phi^{(3)}\left(0\right)\neq 0$. Then solutions of \eqref{eq_16} are of the form $w_j\left(s\right)=a_j\phi\left(s\right)$, $j=1,2,3$. The compatibility condition $w_i\left(0\right)=w_j\left(0\right)$, $i,j=1,2,3$, implies that the $a_i=a_j\equiv a$. Hence
\begin{equation*}
w_j\left(s\right)=a\phi\left(s\right),\quad j=1,2,3.
\end{equation*}
Since $\sum^3_{j=1}w^{(3)}_j\left(0\right)=3a\phi^{(3)}\left(0\right)=0$, we obtain $a=0$. This means that \eqref{eq_16} has only trivial zero solutions. Thus there can be no purely imaginary eigenvalue. This completes the proof of statement \ref{T-3-2-d}, and hence of the theorem.
\end{proof}

\subsection{Eigenvalues and eigenvectors}\label{sec_4}
In this subsection we give a careful analysis of how to determine the eigenvalues of $\mathcal{T}$ and the corresponding eigenvectors. Let $\lambda$ be an eigenvalue of $\mathcal{T}$ with eigenvector $x$. In view of \eqref{eq1sdcxabv},
\begin{equation*}
x=\left\{\left(\begin{matrix}
w_j\\[0.2em]
v_j
\end{matrix}\right)\right\}^3_{j=1}=\left\{\left(\begin{matrix}
w_j\\[0.2em]
\lambda w_j
\end{matrix}\right)\right\}^3_{j=1}
\end{equation*}
where $w_j=w_j\left(s\right)$, $j=1,2,3$, satisfies the boundary-eigenvalue problem
\begin{equation}\label{eq1sdcxabddv}
\left\{\begin{aligned}
w^{(4)}_j-\gamma w''_j&=-\lambda^2 w_j,& j&=1,2,3,\\[0.25em]
w_j\left(1\right)=w''_j\left(1\right)&=0,& j&=1,2,3,\\[0.25em]
w_i\left(0\right)&=w_j\left(0\right),& i,j&=1,2,3,\\[0.25em]
w''_j\left(0\right)-\left(\alpha+\lambda\beta\right) w_j'\left(0\right)&=0,& j&=1,2,3,\\
\sum^3_{j=1}\,(w^{(3)}_j\left(0\right)-\gamma w_j'\left(0\right))&=0.&&&
\end{aligned}\right.
\end{equation}
Then \eqref{eq1sdcxabddv} has associated with it the characteristic equation
\begin{equation}\label{eqwchar42}
\mu^4-\gamma \mu^2+\lambda^2=0.
\end{equation}
The characteristic exponents $\mu=\mu\left(\lambda\right)$ satisfying \eqref{eqwchar42} are
\begin{equation*}
\mu\left(\lambda\right)=\pm\sqrt{\frac{\gamma\pm \sqrt{\gamma^2-4\lambda^2}}{2}},
\end{equation*}
and thus
\begin{equation}\label{eq_19}
\left\{\begin{split}
 \mu_1\left(\lambda\right)&=\displaystyle\sqrt{\frac{\gamma+\sqrt{\gamma^2-4\lambda^2}}{2}},\\[0.2em]
\mu_2\left(\lambda\right)&=\displaystyle-\sqrt{\frac{\gamma+ \sqrt{\gamma^2-4\lambda^2}}{2}}=-\mu_1\left(\lambda\right),\\[0.2em]
\mu_3\left(\lambda\right)&=\displaystyle\sqrt{\frac{\gamma- \sqrt{\gamma^2-4\lambda^2}}{2}},\\[0.2em]
\mu_4\left(\lambda\right)&=\displaystyle-\sqrt{\frac{\gamma- \sqrt{\gamma^2-4\lambda^2}}{2}}=-\mu_3\left(\lambda\right).
\end{split}\right.
\end{equation}
Using \eqref{eq_19} it is easily verified that, provided $\lambda\neq -\frac{\gamma}{2}$ so $\mu_1\neq\mu_3$, the general solution of the differential equation $w^{(4)}-\gamma w''=-\lambda^2 w$ which satisfies the boundary conditions $w\left(1\right)=w''\left(1\right)=0$ is given by 
\begin{equation}\label{eq12xfg}
\psi\left(s\right)=c_1\sinh\mu_1\left(1-s\right)+c_2\sinh\mu_3\left(1-s\right),
\end{equation}
arbitrary constants $c_1$, $c_2$. When $\lambda=-\frac{\gamma}{2}$ we clearly have $\mu_1=\mu_3\equiv \mu$ and we find that
\begin{equation*}
\psi\left(s\right)=c_1\sinh\mu\left(1-s\right)+c_2\,\bigl(s\cosh\mu\left(1-s\right)-e^{\mu\left(1-s\right)}\bigr).
\end{equation*}
Let us agree to restrict attention to the case $\lambda\in\mathbf{C}\backslash\left\{-\frac{\gamma}{2}\right\}$, so we consider \eqref{eq12xfg} instead. In this case we have for the corresponding eigenfunctions of \eqref{eq1sdcxabddv}
\begin{equation}\label{eqeig01}
w_j\left(s\right)=c_{j,1}\sinh\mu_1\left(1-s\right)+c_{j,2}\sinh\mu_3\left(1-s\right),\quad j=1,2,3.
\end{equation}
Using the vertex conditions $w''_j\left(0\right)-\left(\alpha+\lambda\beta\right) w_j'\left(0\right)=0$, $j=1,2,3$, and $\sum^3_{j=1}\,(w^{(3)}_j\left(0\right)-\gamma w_j'\left(0\right))=0$, we then obtain
\begin{align}
c_{j,1}\mu_1\left[\mu_1\sinh\mu_1+\left(\alpha+\lambda\beta\right)\cosh\mu_1\right]+c_{j,2}\mu_3\left[\mu_3\sinh\mu_3+\left(\alpha+\lambda\beta\right)\cosh\mu_3\right]&=0,\label{eq_20ax}\\[0.2em]
\left(\mu^3_1\cosh\mu_1-\gamma \mu_1\cosh \mu_1\right)\sum^3_{j=1}c_{j,1}+\left(\mu^3_3 \cosh\mu_3 -\gamma\mu_3\cosh\mu_3\right)\sum^3_{j=1}c_{j,2}&=0.\label{eq_20bx}
\end{align}
We note from \eqref{eq_19} that
\begin{equation*}
\mu^2_1+\mu^2_3=\gamma,\quad \mu_1\mu_3=\lambda.
\end{equation*}
Substitution in \eqref{eq_20bx} and summing \eqref{eq_20ax} over $j=1,2,3$ yields the following algebraic equations for $\sum^3_{j=1}c_{j,1}$, $\sum^3_{j=1}c_{j,2}$:
\begin{equation}\label{eq_21}
\left\{\begin{split}
\mu_3\cosh\mu_1\sum^3_{j=1}c_{j,1}+ \mu_1\cosh\mu_3\sum^3_{j=1}c_{j,2}&=0,\\[0.2em]
\left[\mu^2_1\sinh\mu_1+\left(\alpha+\lambda\beta\right)\mu_1\cosh\mu_1\right]\sum^3_{j=1}c_{j,1}\hspace{0.2\linewidth}&\\
+\left[\mu^2_3\sinh\mu_3+\left(\alpha+\lambda\beta\right)\mu_3\cosh\mu_3\right]\sum^3_{j=1}c_{j,2}&=0.
\end{split}\right.
\end{equation}
We divide the discussion of the solutions of \eqref{eq_21} into two cases. Before going on to these, however, it will be convenient to derive a few intermediate results. Let $\Delta_1=\Delta_1\left(\lambda\right)$ be the determinant of the coefficient matrix formed by \eqref{eq_21}. We then see easily that
\begin{equation}\label{eq_22}
\Delta_1\left(\lambda\right)=\mu^3_3\cosh\mu_1\sinh\mu_3-\mu^3_1\sinh\mu_1\cosh\mu_3
+(\mu^2_3-\mu^2_1)\left(\alpha+\lambda\beta\right)\cosh\mu_1\cosh\mu_3.
\end{equation}
Let us set
\begin{equation}\label{eqcase12xa}
\sum^3_{j=1}c_{j,2}=-b\mu_1\left[\mu_1\sinh\mu_1+\left(\alpha+\lambda\beta\right)\cosh\mu_1\right],
\end{equation}
for some constant $b\neq 0$. Then
\begin{equation}\label{eqcase12xb}
\sum^3_{j=1}c_{j,1}=b \mu_3\left[\mu_3\sinh\mu_3+\left(\alpha+\lambda\beta\right)\cosh\mu_3\right].
\end{equation}
The compatibility condition $w_i\left(0\right)=w_j\left(0\right)\equiv w\left(0\right)$, $i,j=1,2,3$, implies, together with \eqref{eqeig01}, that
\begin{equation}\label{eqnew08}
c_{j,1}\sinh\mu_1+c_{j,2}\sinh\mu_3=w\left(0\right),\quad j=1,2,3.
\end{equation}
So it follows on summing \eqref{eqnew08} over $j=1,2,3$ and using \eqref{eqcase12xa},  \eqref{eqcase12xb} that
\begin{align*}
&\sinh\mu_1\left\{b \mu_3\left[\mu_3\sinh\mu_3+\left(\alpha+\lambda\beta\right)\cosh\mu_3\right]\right\}\\
&\hspace{0.2\linewidth}-\sinh\mu_3\left\{b\mu_1\left[\mu_1\sinh\mu_1+\left(\alpha+\lambda\beta\right)\cosh\mu_1\right]\right\}=3w\left(0\right).
\end{align*}
From this it follows that if we let $\Delta_2=\Delta_2\left(\lambda\right)$ be another determinant given by
\begin{equation}\label{eq_23}
\Delta_2\left(\lambda\right)=(\mu^2_3-\mu^2_1)\sinh\mu_1\sinh\mu_3+\left(\alpha+\lambda\beta\right)\left(\mu_3\sinh\mu_1\cosh\mu_3 -\mu_1\cosh\mu_1\sinh\mu_3\right),
\end{equation}
then
\begin{equation*}
w\left(0\right)=\frac{b}{3}\Delta_2\left(\lambda\right).
\end{equation*}
\begin{enumerate}[label=,leftmargin=*,align=left,labelwidth=\parindent,labelsep=0pt]
\item\label{item04}\textbf{Case 1.} We consider first the case of a nonzero solution pair $\sum^3_{j=1}c_{j,1}$, $\sum^3_{j=1}c_{j,2}$ of \eqref{eq_21}. The necessary and sufficient condition for this is that $\Delta_1\left(\lambda\right)=0$. Solving the system of equations given by \eqref{eq_20ax} and \eqref{eqnew08} for $c_{j,1}$, $c_{j,2}$, $j=1,2,3$, we have
\begin{align}
 c_{j,1}&=\displaystyle\frac{b\mu^2_3\left[\mu_3\sinh\mu_3+\left(\alpha+\lambda\beta\right)\cosh\mu_3\right]\cosh\mu_1}{3\mu_3\cosh\mu_1},\label{eqnew05ax}\\[0.2em]
c_{j,2}&=\displaystyle-\frac{b\mu^2_1\left[\mu_1\sinh\mu_1+\left(\alpha+\lambda\beta\right)\cosh\mu_1\right]\cosh\mu_3}{3\mu_1\cosh\mu_3}.\label{eqnew05axss}
\end{align}
Note that
\begin{equation*}
\mu^2_3\left[\mu_3\sinh\mu_3+\left(\alpha+\lambda\beta\right)\cosh\mu_3\right]\cosh\mu_1
=\Delta_1\left(\lambda\right)+\mu^2_1\left[\mu_1\sinh\mu_1+\left(\alpha+\lambda\beta\right)\cosh\mu_1\right]\cosh\mu_3.
\end{equation*}
Let 
\begin{equation*}
b=\frac{3c \mu_1\mu_3\cosh\mu_1\cosh\mu_3}{\mu^2_1\left[\mu_1\sinh\mu_1+\left(\alpha+\lambda\beta\right)\cosh\mu_1\right]\cosh\mu_3}.
\end{equation*}
It is easy to see then from \eqref{eqnew05ax}, \eqref{eqnew05axss} that $c_{j,1}= c\mu_1\cosh\mu_3$, $c_{j,2}=-c\mu_3\cosh\mu_1$, $j=1,2,3$. Using these relations in \eqref{eqeig01}, we obtain for the eigenfunctions
\begin{equation}\label{eqwwedf4}
w_j\left(s\right)=c\left(\mu_1\cosh\mu_3\sinh\mu_1\left(1-s\right)-\mu_3\cosh\mu_1\sinh\mu_3\left(1-s\right)\right),\quad j=1,2,3.
\end{equation}
\item\label{item05}\textbf{Case 2.} We suppose now that \eqref{eq_21} has only the zero solution, i.e.\ $\Delta_1\left(\lambda\right)\neq 0$.
From \eqref{eqnew08} we then get that $w\left(0\right)=0$. Nonzero solutions of the system of equations given by \eqref{eq_20ax} and \eqref{eqnew08} exist if and only if $\Delta_2\left(\lambda\right)=0$. Then, because $\sinh\mu_1\neq 0$ and $\sinh\mu_3\neq 0$, we have
\begin{equation*}
c_{j,1}=-\frac{\sinh\mu_3}{\sinh\mu_1}\,c_{j,2},\quad j=1,2,3.
\end{equation*}
Setting
\begin{equation*}
a_j=\frac{c_{j,1}}{\sinh\mu_3}=-\frac{c_{j,2}}{\sinh\mu_1},\quad j=1,2,3,
\end{equation*}
and recalling that here $\sum^3_{j=1}c_{j,1}=\sum^3_{j=1}c_{j,2}=0$, we arrive at
\begin{equation*}
\sum^3_{j=1}a_j=\frac{1}{\sinh\mu_3}\sum^3_{j=1}c_{j,1}=-\frac{1}{\sinh\mu_1}\sum^3_{j=1}c_{j,2}=0.
\end{equation*}
In this case \eqref{eqeig01} becomes
\begin{equation*}
w_{j}\left(s\right)=a_j\left(\sinh\mu_3\sinh\mu_1\left(1-s\right)-\sinh\mu_1\sinh\mu_3\left(1-s\right)\right),\quad j=1,2,3,\quad \sum^3_{j=1}a_j=0.
\end{equation*}
Let us write $w_{1,j}= a_{1,j}\phi$, $w_{2,j}= a_{2,j}\phi$, $j=1,2,3$, where 
\begin{equation}\label{eq19}
\left\{a_{1,1},a_{1,2},a_{1,3}\right\}=\left\{-2,1,1\right\},\quad \left\{a_{2,1},a_{2,2},a_{2,3}\right\}=\left\{0,1,-1\right\}
\end{equation}
and $\phi$ is given by
\begin{equation}\label{eq19phi}
\phi\left(s\right)=\sinh\mu_3\sinh\mu_1\left(1-s\right)-\sinh\mu_1\sinh\mu_3\left(1-s\right).
\end{equation}
The eigenvector $x$ corresponding to $\lambda$ can then be represented as
\begin{equation*}
x= \hat{a}x_1+\hat{c}x_2,\quad \hat{a},\hat{c}\in \mathbf{C},
\end{equation*}
with
\begin{equation*}
x_1=\left\{\left(\begin{matrix}
w_{1,j}\\[0.2em]
\lambda w_{1,j}
\end{matrix}\right)\right\}^3_{j=1},\quad 
x_2=\left\{\left(\begin{matrix}
w_{2,j}\\[0.2em]
\lambda w_{2,j}
\end{matrix}\right)\right\}^3_{j=1}.
\end{equation*}
\end{enumerate}
We summarise in the following result.
\begin{theorem}\label{T-4-1}
Let the characteristic exponents satisfying \eqref{eqwchar42} be given as in \eqref{eq_19} for $\lambda\in\mathbf{C}\backslash\left\{-\frac{\gamma}{2}\right\}$, and let the determinants $\Delta_1\left(\lambda\right)$ and $\Delta_2\left(\lambda\right)$ be given by \eqref{eq_22} and \eqref{eq_23}, respectively. The spectrum of $\mathcal{T}$ consists of two branches,
\begin{equation*}
\sigma\left(\mathcal{T}\right)=\sigma^{(1)}\left(\mathcal{T}\right)\cup\sigma^{(2)}\left(\mathcal{T}\right),
\end{equation*}
where
\begin{equation*}
\sigma^{(1)}\left(\mathcal{T}\right)=\left\{\lambda\in \mathbf{C}\backslash\left\{-\frac{\gamma}{2}\right\}~\middle|
~\Delta_1\left(\lambda\right)=0\right\},\quad \sigma^{(2)}\left(\mathcal{T}\right)=\left\{\lambda\in \mathbf{C}\backslash\left\{-\frac{\gamma}{2}\right\}~\middle|
~\Delta_2\left(\lambda\right)=0\right\}.
\end{equation*}
For each $\lambda\in \sigma^{(1)}\left(\mathcal{T}\right)$ the corresponding eigenvector is
\begin{equation*}
x^{(1)}=\left\{\left(\begin{matrix}
w_j\\[0.2em]
\lambda w_j
\end{matrix}\right)\right\}^3_{j=1},
\end{equation*}
where the $w_j$ are given by \eqref{eqwwedf4}. For each $\lambda\in\sigma^{(2)}\left(\mathcal{T}\right)$ the corresponding eigenvector is of the form
\begin{equation*}
x^{(2)}=\hat{a}x_1+\hat{c}x_2,\quad \hat{a},\hat{c}\in \mathbf{C},
\end{equation*}
with
\begin{equation*}
x_1=\left\{\left(\begin{matrix}
a_{1,j}\phi\\[0.2em]
 a_{1,j}\lambda\phi
\end{matrix}\right)\right\}^3_{j=1},\quad 
x_2=\left\{\left(\begin{matrix}
a_{2,j}\phi\\[0.2em]
 a_{2,j}\lambda\phi
\end{matrix}\right)\right\}^3_{j=1},
\end{equation*}
where the $a_{1,j}$, $a_{2,j}$ and $\phi$ are given by \eqref{eq19} and \eqref{eq19phi}, respectively.
\end{theorem}

\begin{remark}
With the choice of eigenvectors $x_1$, $x_2$ as in the theorem, one can verify directly that they form an orthogonal basis for the eigensubspace. Indeed, taking the inner product of $x_1$ with $x_2$ yields
\begin{align*}
\<x_1,x_2>&=\sum^3_{j=1}a_{1,j}\overline{a_{2,j}}\left(\int^1_0\,\bigl|\phi''\left(s\right)\bigr|^2ds+\alpha\,\bigl|\phi'\left(0\right)\bigr|^2\right)+\gamma\sum^3_{j=1}a_{1,j}\overline{a_{2,j}}\int^1_0\,\bigl|\phi'\left(s\right)\bigr|^2ds\\
&\qquad+\sum^3_{j=1}a_{1,j}\overline{a_{2,j}}\int^1_0\,\bigl|\lambda\phi\left(s\right)\bigr|^2ds\\[0.2em]
&=\left(\int^1_0\,\bigl|\phi''\left(s\right)\bigr|^2ds+\alpha\,\bigl|\phi'\left(0\right)\bigr|^2+\gamma\int^1_0\,\bigl|\phi'\left(s\right)\bigr|^2ds+\int^1_0\,\bigl|\lambda\phi\left(s\right)\bigr|^2ds\right)\sum^3_{j=1}a_{1,j}\overline{a_{2,j}}
\end{align*}
which in view of the choice of constants $a_{1,j}$, $a_{2,j}$ as in \eqref{eq19} obviously implies $x_1\perp x_2$.
\end{remark}

\subsection{Eigenvalue asymptotics}\label{sec_5}
Our main task in this subsection will be to investigate, by analysis of the asymptotic zeros of the determinants $\Delta_1$ and $\Delta_2$ given by \eqref{eq_22} and \eqref{eq_23}, respectively, how the two branches $\sigma^{(1)}\left(\mathcal{T}\right)$ and $\sigma^{(2)}\left(\mathcal{T}\right)$ of eigenvalues of $\mathcal{T}$ behave asymptotically. The standard procedure for such an undertaking originated in the work of Birkhoff \cite{Birkhoff1908a,Birkhoff1908b} (and was further developed by Naimark, see \cite{Naimark1967}). Take, e.g., the branch $\sigma^{(1)}\left(\mathcal{T}\right)$ (the branch $\sigma^{(2)}\left(\mathcal{T}\right)$ could be treated similarly). The procedure involves the use of asymptotic expansions of $\Delta_1\left(\lambda\right)$ when $\left|\lambda\right|$ is large and subsequent use of Rouche's theorem, applied to an appropriately chosen comparison function (whose zeros, or at least their proper enumeration and asymptotics are known), to count the number of the zeros of $\lambda\mapsto\Delta_1\left(\lambda\right)$ in small neighbourhoods, say disks, around the zeros of the corresponding comparison function and verify that in each such disc there is exactly one zero of $\Delta_1$. An advantage of this method is that it requires no closed-form computation of $\Delta_1$. However, when the determinants may be written down in closed form, as in the cases here, one is led to adopt a more direct approach. In this vein, we use the explicit forms \eqref{eq_22} and \eqref{eq_23} and then proceed to derive the asymptotic zeros by means of a direct, but involved, computation.

In the course of our analysis it will frequently be useful to write
\begin{equation*}
{\mu}_3-{\mu}_1=\left(i-1\right)\rho,\quad {\mu}_1+{\mu}_3=\left(1+i\right)\rho
\end{equation*}
for the (yet to be determined) parameter $\rho$, leading to the relationships
\begin{equation}\label{eqrel023}
\mu_3=i{\rho},\quad\mu^2_3=-{\rho}^2,\quad\mu^3_3=-i{\rho}^3,\quad
\mu_1={\rho},\quad\mu^2_1={\rho}^2,\quad\mu^3_1={\rho}^3,
\end{equation}
and note that
\begin{equation}\label{eqrel023ss}
 \lambda=\mu_1\mu_3=i{\rho}^2.
\end{equation}
We know from Theorem \ref{T-3-2} that the eigenvalues of $\mathcal{T}$ lie in the left half-plane and occur in symmetric pairs $\lambda$, $\overline{\lambda}$, corresponding to those eigenvalues with nonzero imaginary part. So let us consider the sector $\arg\lambda\in\left[\frac{\pi}{2}, \pi\right]$ and, with the substitution \eqref{eqrel023ss}, thus consider $\arg\rho\in\left[0, \frac{\pi}{4}\right]$ in what follows.

\subsubsection{Asymptotic zeros of $\Delta_1$}\label{sec_5_1}
The following two elementary formulae will be needed: 
\begin{align*}
\sinh\left(x+y\right)=\sinh x\cosh y+\cosh x\sinh y,\\[0.2em]
\cosh\left(x+y\right)=\cosh x\cosh y+\sinh x \sinh y.
\end{align*}
Let us use these to rewrite \eqref{eq_22} as
\begin{equation}\label{eq_22x}
\begin{split}
2\Delta_1\left(\lambda\right)&=(\mu^3_3-\mu^3_1)\sinh\left(\mu_1+\mu_3\right)+(\mu^3_1+\mu^3_3)\sinh(\mu_3-\mu_1) \\
&\qquad+(\mu^2_3-\mu^2_1)\left(\alpha+\lambda\beta\right)\left(\cosh(\mu_1+\mu_3)+\cosh(\mu_3-\mu_1)\right).
\end{split}
\end{equation}
The asymptotic zeros of $\Delta_1$ are investigated in three steps.
\begin{enumerate}[label=,leftmargin=*,align=left,labelwidth=\parindent,labelsep=0pt,ref={\arabic*}]
\item\label{item06}\textbf{Step 1.} We first determine the roots of the equation
\begin{equation*}
\cosh\left(1+i\right)\rho+\cosh\left(i-1\right)\rho=0.
\end{equation*}
We may equivalently write that
\begin{equation*}
-1=\frac{\cosh\left(i-1\right)\rho}{\cosh\left(1+i\right)\rho},
\end{equation*}
which in the sector $\arg\rho\in\left[0, \frac{\pi}{4}\right]$ has the roots
\begin{equation*}
\rho_n=\left(n +\frac{1}{2}\right)\pi.
\end{equation*}
So in this case we have
\begin{equation}\label{eq18}
\sinh\left(i-1\right)\rho_n=\sinh\left(1+i\right)\rho_n,\quad \cosh\left(i-1\right)\rho_n=-\cosh\left(1+i\right)\rho_n.
\end{equation}
\item\label{item07}\textbf{Step 2.} Let the sequence $\left\{\lambda_n\right\}$ represent the zeros of $\Delta_1$, where the $\lambda_n$ are considered in the neighbourhoods of the $i\rho_n^2$. Setting
\begin{equation*}
\tilde{\rho}_n=\rho_n+\xi_n,
\end{equation*}
where the $\xi_n$ will be determined, writing $\tilde{\rho}_n$ in place of $\rho$ in \eqref{eqrel023} and \eqref{eqrel023ss}, and substituting the relevant relationships in \eqref{eq_22x} we get
\begin{align*}
2\Delta_1\left(\lambda_n\right)&=-\tilde{\rho}^3_n\left(1+i\right)\sinh\left(1+i\right)\tilde{\rho}_n -\tilde{\rho}^3_n\left(i-1\right)\sinh\left(i-1\right)\tilde{\rho}_n \\
&\qquad-2\tilde{\rho}^2_n\,(\alpha+i\tilde{\rho}^2_n\beta)\left(\cosh\left(1+i\right)\tilde{\rho}_n+\cosh\left(i-1\right)\tilde{\rho}_n\right).
\end{align*}
The equation $2\Delta_1\left(\lambda_n\right)=0$ then becomes
\begin{align*}
0&=-\tilde{\rho}_n\left(1+i\right)\sinh\left(1+i\right)\tilde{\rho}_n-\tilde{\rho}_n\left(i-1\right)\sinh\left(i-1\right)\tilde{\rho}_n \\
&\qquad -2\,(\alpha+i\tilde{\rho}^2_n\beta)\left(\cosh\left(1+i\right)\tilde{\rho}_n+\cosh\left(i-1\right)\tilde{\rho}_n\right)
\end{align*}
or, equivalently,
\begin{equation}\label{eq_24}
-\frac{\tilde{\rho}_n}{2\,(\alpha+i\tilde{\rho}^2_n\beta)} =\frac{\cosh\left(1+i\right)\tilde{\rho}_n+\cosh\left(i-1\right)\tilde{\rho}_n}{\left(1+i\right)\sinh\left(1+i\right)\tilde{\rho}_n+\left(i-1\right)\sinh\left(i-1\right)\tilde{\rho}_n}.
\end{equation}
We now use \eqref{eq_24} to determine the asymptotics of $\xi_n$. For $\rho_n$ large we can simplify our calculations below by noting that
\begin{equation*}
\tanh\left(1+i\right)\rho_n=\frac{\sinh\left(1+i\right)\rho_n}{\cosh\left(1+i\right)\rho_n}
=1+{O}\,(e^{-2\rho_n})
\end{equation*}
so
\begin{equation}\label{eqnew321}
\tanh\left(1+i\right)\rho_n\sim 1.
\end{equation}
We proceed now to estimate the right-hand side of \eqref{eq_24}. Consider first the numerator. A straightforward calculation using \eqref{eq18} and \eqref{eqnew321} shows that for large $\rho_n$
\begin{align*}
&\cosh\left(1+i\right)\tilde{\rho}_n+\cosh\left(i-1\right)\tilde{\rho}_n\notag\\[0.2em]
&\hspace{0.1\linewidth}\sim\cosh\left(1+i\right)\rho_n\left(\cosh\left(1+i\right)\xi_n-\cosh\left(i-1\right)\xi_n\right.\\
&\hspace{0.4\linewidth}\left.+\sinh\left(1+i\right)\xi_n+\sinh\left(i-1\right)\xi_n\right)\notag\\[0.2em]
&\hspace{0.1\linewidth}= 2e^{\xi_n}\sinh i\xi_n\cosh\left(1+i\right)\rho_n.\notag
\end{align*}
Likewise for the denominator in the right-hand side of \eqref{eq_24},
\begin{align*}
&\left(1+i\right)\sinh\left(1+i\right)\tilde{\rho}_n+\left(i-1\right)\sinh\left(i-1\right)\tilde{\rho}_n\notag\\[0.2em]
&\hspace{0.1\linewidth}\sim\cosh\left(1+i\right)\rho_n\left[\left(1+i\right)\left(\sinh\left(1+i\right)\xi_n+\cosh\left(1+i\right)\xi_n\right)\right.\\
&\hspace{0.4\linewidth}\left.-\left(i-1\right)\left(\sinh\left(i-1\right)\xi_n-\cosh\left(i-1\right)\xi_n\right)\right]\notag\\[0.2em]
&\hspace{0.1\linewidth}= 2e^{\xi_n}\left(\sinh i\xi_n+i\cosh i\xi_n\right)\cosh\left(1+i\right)\rho_n.\notag
\end{align*}
So the right-hand side of \eqref{eq_24} becomes
\begin{equation*}
\frac{\cosh\left(1+i\right)\tilde{\rho}_n+\cosh\left(i-1\right)\tilde{\rho}_n}{\left(1+i\right)\sinh\left(1+i\right)\tilde{\rho}_n+\left(i-1\right)\sinh\left(i-1\right)\tilde{\rho}_n}\sim \frac{\sinh i\xi_n}{\sinh i\xi_n+i\cosh i\xi_n}
\end{equation*}
from which it is clear that $\xi_n$ is bounded for large $ n$. In fact, for large $ n$ we can now write \eqref{eq_24} as
\begin{equation}\label{eq_24new}
-\frac{\tilde{\rho}_n}{2\,(\alpha+i\tilde{\rho}^2_n\beta)}\sim \xi_n.
\end{equation}
To see how the left-hand side of \eqref{eq_24new} behaves asymptotically we observe that
\begin{equation*}
\frac{\tilde{\rho}_n}{\alpha+i\tilde{\rho}^2_n\beta}=-\frac{i}{\tilde{\rho}_n\beta}\left[1+\frac{1}{\tilde{\rho}_n^2}\left(\frac{i\alpha}{\beta}\right)+\frac{1}{\tilde{\rho}_n^4}\left(\frac{i\alpha}{\beta}\right)^2+\cdots\right]=-\frac{i}{\tilde{\rho}_n\beta}+{O}\,(\tilde{\rho}^{-3}_n),
\end{equation*}
wherein
\begin{equation*}
\frac{1}{\tilde{\rho}_n}=\frac{1}{\rho_n+\xi_n}=\frac{1}{\rho_n}\left[1+\left(-\frac{\xi}{\rho_n}\right)+\left(-\frac{\xi}{\rho_n}\right)^2+\cdots\right].
\end{equation*}
We may, therefore, replace \eqref{eq_24new} by
\begin{equation*}
\frac{i}{2\rho_n\beta}-\frac{i\xi_n}{2\rho_n^2\beta}+{O}\,(\rho^{-3}_n)\sim \xi_n
\end{equation*}
which leads to
\begin{align*}
 \xi_n\sim \frac{i\rho_n}{i+2\rho_n^2\beta}+{O}\,(\rho^{-3}_n)&=\frac{i}{2\rho_n\beta}\left[1+\frac{1}{{\rho}_n^2}\left(-\frac{i}{2\beta}\right)+\frac{1}{{\rho}_n^4}\left(-\frac{i}{2\beta}\right)^2+\cdots\right]+{O}\,(\rho^{-3}_n)\\[0.2em]
&=\frac{i}{2\rho_n\beta}+{O}\,(\rho^{-3}_n).
\end{align*}
\item\label{item08}\textbf{Step 3.} An easy calculation, taking into account that $\lambda_n=i\tilde{\rho}^2_n=i\left(\rho_n+\xi_n\right)^2$, shows that the eigenvalues belonging to the first branch $\sigma^{(1)}\left(\mathcal{T}\right)$ of the spectrum have asymptotic representations
\begin{equation}\label{eq_24x}
\lambda_n=-\frac{1}{\beta}+i\rho^2_n+{O}\,(\rho^{-2}_n),\quad \rho_n=\left(n +\frac{1}{2}\right)\pi,
\end{equation}
for large $ n$.
\end{enumerate}

\subsubsection{Asymptotic zeros of $\Delta_2$}\label{sec_5_2}
We turn now to the asymptotic zeros of $\Delta_2$. Again we proceed in three steps.
\begin{enumerate}[label=,leftmargin=*,align=left,labelwidth=\parindent,labelsep=0pt,ref={\arabic*}]
\item\label{item09}\textbf{Step 1.} First the roots of the equation
\begin{equation}\label{eqnew323}
i \sinh\rho\cosh i\rho - \cosh\rho\sinh i\rho=0
\end{equation}
are determined. A simple manipulation using
\begin{align*}
2i \sinh\rho\cosh i\rho -2 \cosh\rho\sinh i\rho&=i\left(\sinh\left(1+i\right)\rho-\sinh\left(i-1\right)\rho\right)\\
&\qquad-\sinh\left(1+i\right)\rho-\sinh\left(i-1\right)\rho\\[0.2em]
&=\left(i-1\right)\sinh\left(1+i\right)\rho-\left(1+i\right)\sinh\left(i-1\right)\rho
\end{align*}
shows that \eqref{eqnew323} is equivalent to
\begin{equation*}
i=\frac{\sinh\left(i-1\right)\rho}{\sinh\left(1+i\right)\rho}=-\frac{e^{-2i\rho}-e^{-2\rho} }{1 -e^{-2\left(1+i\right)\rho}}= -e^{-2i\rho}+{O}\,(e^{-2\rho}).
\end{equation*}
The roots are
\begin{equation*}
\rho_n=\left(n+\frac{1}{4}\right)\pi+\nu_n,
\end{equation*}
where $\nu_n={O}\,(e^{-2\rho_n})$. In this case
\begin{equation}\label{eq27}
\frac{\sinh\left(i-1\right)\rho_n}{\sinh\left(1+i\right)\rho_n}=i,\quad \frac{\cosh\left(i-1\right)\rho_n}{\cosh\left(1+i\right)\rho_n}= -i\,(1+{O}\,(e^{-2\rho_n})).
\end{equation}
On the other hand, we know that
\begin{equation}\label{eq27new}
\tanh\left(1+i\right)\rho_n=1+{O}\,(e^{-2\rho_n}).
\end{equation}
In the next step we will use \eqref{eq27} and \eqref{eq27new} in our calculations.
\item\label{item10}\textbf{Step 2.} Let the sequence $\left\{\lambda_n\right\}$ represent the zeros of $\Delta_2$. We proceed as in Step \ref{item07} of the previous subsection, setting
\begin{equation*}
\tilde{\rho}_n=\rho_n+\xi_n,
\end{equation*}
writing $\tilde{\rho}_n$ in place of $\rho$ in \eqref{eqrel023}, and using the relationships in \eqref{eq_23} to obtain
\begin{equation*}
\Delta_2\left(\lambda_n\right)=-2\tilde{\rho}^2_n\sinh\tilde{\rho}_n\sinh i\tilde{\rho}_n+(\alpha+i\tilde{\rho}^2_n\beta)
\left(i\tilde{\rho}_n\sinh\tilde{\rho}_n\cosh i\tilde{\rho}_n -\tilde{\rho}_n\cosh\tilde{\rho}_n\sinh i\tilde{\rho}_n \right).
\end{equation*}
The equation $2\Delta_2\left(\lambda_n\right)=0$ then reads
\begin{equation*}
-4\tilde{\rho}_n\sinh\tilde{\rho}_n\sinh i\tilde{\rho}_n+(\alpha+i\tilde{\rho}^2_n\beta)
\left(2i \sinh\tilde{\rho}_n\cosh i\tilde{\rho}_n -2\cosh\tilde{\rho}_n\sinh i\tilde{\rho}_n \right)=0,
\end{equation*}
which is equivalent to
\begin{equation}\label{eq_24b}
\frac{2\tilde{\rho}_n}{\alpha+i\tilde{\rho}^2_n \beta}
=\frac{2i\sinh\tilde{\rho}_n\cosh i\tilde{\rho}_n-2\cosh\tilde{\rho}_n\sinh i\tilde{\rho}_n}
{2\sinh\tilde{\rho}_n\sinh i\tilde{\rho}_n}.
\end{equation}
Let us estimate the right-hand side of \eqref{eq_24b} and begin with the numerator. Using \eqref{eq27} and \eqref{eq27new}, we obtain after straightforward calculations
\begin{align*}
&2i\sinh\tilde{\rho}_n\cosh i\tilde{\rho}_n-2\cosh\tilde{\rho}_n\sinh i\tilde{\rho}_n\\[0.2em]
&\hspace{0.1\linewidth}=\left(i-1\right)\cosh\left(1+i\right)\rho_n\left(\cosh\left(1+i\right)\xi_n-\cosh\left(i-1\right)\xi_n\right.\\
&\hspace{0.4\linewidth}\left.+\sinh\left(1+i\right)\xi_n+\sinh\left(i-1\right)\xi_n\right)+{O}\,(e^{-\rho_n})\\[0.2em]
&\hspace{0.1\linewidth}=2e^{\xi_n}\left(i-1\right)\sinh i\xi_n\cosh\left(1+i\right)\rho_n+{O}\,(e^{-\rho_n}),
\end{align*}
Likewise, considering the denominator in the right-hand side of \eqref{eq_24b},
\begin{align*}
&2\sinh\tilde{\rho}_n\sinh i\tilde{\rho}_n\\[0.2em]
&\hspace{0.1\linewidth}=\cosh\left(1+i\right)\rho_n\left(\cosh\left(1+i\right)\xi_n+i\cosh\left(i-1\right)\xi_n\right.\\
&\hspace{0.4\linewidth}\left.+\sinh\left(1+i\right)\xi_n-i\sinh\left(i-1\right)\xi_n\right)+{O}\,(e^{-\rho_n})\\[0.2em]
&\hspace{0.1\linewidth}=-e^{\xi_n}\left(i-1\right)\left(\sinh i\xi_n+i\cosh i\xi_n\right)\cosh\left(1+i\right)\rho_n+{O}\,(e^{-\rho_n}).
\end{align*}
Thus, for large $\rho_n$, the right-hand side of \eqref{eq_24b} becomes
\begin{equation*}
\frac{2i\sinh\tilde{\rho}_n\cosh i\tilde{\rho}_n-2\cosh\tilde{\rho}_n\sinh i\tilde{\rho}_n}
{2\sinh\tilde{\rho}_n\sinh i\tilde{\rho}_n} \sim -\frac{2\sinh i\xi_n}{\sinh i\xi_n+i\cosh i\xi_n},
\end{equation*}
and so, for large $ n$, we obtain by arguments similar to those given in the previous subsection (in Step \ref{item07})
\begin{equation*}
\xi_n \sim\frac{i}{\rho_n\beta}+{O}\,(\rho^{-3}_n).
\end{equation*}
\item\label{item11}\textbf{Step 3.} A calculation entirely analogues to that in Step \ref{item08} leading to \eqref{eq_24x} shows that the eigenvalues belonging to the second branch $\sigma^{(2)}\left(\mathcal{T}\right)$ of the spectrum have asymptotic representations
\begin{equation*}
\lambda_n=-\frac{2}{\beta}+i\rho^2_n+{O}\,(\rho^{-2}_n),\quad \rho_n=\left(n+\frac{1}{4}\right)\pi+\nu_n,\quad \nu_{ n }={O}\,(e^{-2 n \pi}),
\end{equation*}
for large $ n$.
\end{enumerate}

The results developed in this section so far, together with Theorems \ref{T-3-2} and \ref{T-4-1}, prove the following theorem.
\begin{theorem}\label{T-5-1}
The spectrum of $\mathcal{T}$ consists of two branches of an infinite sequence of normal eigenvalues, each sequence being symmetric about the real axis and accumulating only at infinity. Asymptotically, for large $n$, the eigenvalues of $\mathcal{T}$ split into the two branches of eigenvalues
\begin{equation*}
\sigma\left(\mathcal{T}\right)=\{\lambda^{(1)}_{\pm n}{\}}^\infty_{n=0}\cup\{\lambda^{(2)}_{\pm n}{\}}^\infty_{n=0}
\end{equation*}
which can be properly enumerated (see Remark \ref{remark02w}) such that 
\begin{equation*}
\operatorname{Im}\lambda^{(2)}_0<\operatorname{Im}\lambda^{(1)}_0<\operatorname{Im}\lambda^{(2)}_1<\operatorname{Im}\lambda^{(1)}_1<\cdots<\operatorname{Im}\lambda^{(2)}_n<\operatorname{Im}\lambda^{(1)}_n<\operatorname{Im}\lambda^{(2)}_{n+1}<\operatorname{Im}\lambda^{(1)}_{n+1}<\cdots
\end{equation*}
and which, as $n \rightarrow\infty$, satisfy
\begin{equation*} 
\lambda^{(1)}_{\pm n}=-\frac{1}{\beta}\pm i\,(\rho^{(1)}_{ n }{)}^2+{O}\,\bigl((\rho^{(1)}_{ n }{)}^{-2}\bigr),\quad \rho^{(1)}_{ n }=\left( n  +\frac{1}{2}\right)\pi,
\end{equation*}
and
\begin{equation*} 
\lambda^{(2)}_{\pm n}=-\frac{2}{\beta}\pm i\,(\rho^{(2)}_{ n }{)}^2+{O}\,\bigl((\rho^{(2)}_{ n }{)}^{-2}\bigr),\quad \rho^{(2)}_{ n }=\left( n  +\frac{1}{4}\right)\pi+\nu_{ n },\quad \nu_{ n }={O}\,(e^{-2 n \pi}).
\end{equation*}
\end{theorem}

\begin{remark}\label{remark02w}
The enumeration of eigenvalues is said to be \textit{proper} if (i) the algebraic multiplicities are taken into account; (ii) ${\lambda}_{-n}=\overline{\lambda_n}$ whenever $\operatorname{Im}{\lambda}_n\neq 0$; (iii) $ 0\leq \operatorname{Im}{\lambda}_n\leq\operatorname{Im}{\lambda}_{n+1}$; and (iv) there are two numbers ${\lambda}_{+ 0}$ and ${\lambda}_{- 0}$. Such an enumeration gives a sequence $\left\{{\lambda}_{ n}\right\}$ with $\left|\operatorname{Re}\lambda_{ n}\right|\leq N<\infty$, some constant $N$, and no finite accumulation points, i.e., $\operatorname{Im}\lambda_{n}\rightarrow\infty$, as $n\rightarrow\infty$.
\end{remark}

\begin{remark}
We have obtained a strengthening of Remark \ref{rem01245} that is of independent interest: the elasticity parameter $\alpha$ has no effect on the eigenvalue asymptotics, as it is at most within ${O}\,(n^{-2})$. So in terms of proving exponential stability, it makes no difference what value is chosen for $\alpha$.
\end{remark}

\subsection{Multiplicity of eigenvalues}\label{sec_5_3}
In this final subsection we determine the multiplicities of the eigenvalues of $\mathcal{T}$. For this it is useful to consider the adjoint spectral problem of \eqref{eq_01xaa02} ($^*$ denoting adjoint)
\begin{equation}\label{eq_14xcabaa}
\mathcal{T}^*z=\overline{\lambda} z,\quad z\in\bm{D}\left(\mathcal{T}^*\right)\subset \mathbb{X},\quad \lambda\in\mathbf{C},
\end{equation}
for then the method of proof used in \cite[Corollary 4.2.2]{Xu2010} may be followed to show that for any eigenvector $x$ associated with an eigenvalue $\lambda$ of $\mathcal{T}$ \textit{there is a nonzero element $z\in\operatorname{Ker}\,(\overline{\lambda} I-\mathcal{T}^*)$ such that $\<x,z> \neq 0$. Thus $\lambda$ is a simple, respectively semisimple, eigenvalue if $\operatorname{dim}\operatorname{Ker}\left(\lambda I-\mathcal{T}\right)=1$, respectively $\operatorname{dim}\operatorname{Ker}\left(\lambda I-\mathcal{T}\right)\neq 1$.}

The following characterisation of the adjoint of $\mathcal{T}$ will be required in the derivation to follow.
\begin{proposition}\label{T-5-2}
The adjoint of the system operator $\mathcal{T}$, as defined by \eqref{eq_08}, \eqref{eq_07}, is the operator $\mathcal{T}^*$ with domain
\begin{equation*}
\bm{D}\left(\mathcal{T}^*\right)=\left\{z=\left\{z_j\right\}^3_{j=1}\in \mathbb{X}~\middle|
~\begin{gathered}
z_j=\left(\begin{matrix}
\tilde{w}_j\\[0.2em]
\tilde{v}_j
\end{matrix}\right)\in (\bm{H}^4\left(0,1\right)\cap \hat{\bm{H}}^2\left(0,1\right))\times \hat{\bm{H}}^2\left(0,1\right), \\[0.2em]
\tilde{w}''_j\left(1\right)=0,\quad \tilde{w}''_j\left(0\right)-\alpha \tilde{w}'_j\left(0\right)+\beta \tilde{v}'_j\left(0\right)=0,\quad j=1,2,3,\\
\sum_{j=1}^3\,(\tilde{w}^{(3)}_j\left(0\right)-\gamma  \tilde{w}'_j\left(0\right))=0
\end{gathered}\right\},
\end{equation*}
defined by
\begin{equation*}
\mathcal{T}^*{z}\coloneqq\left\{\left(\begin{matrix}
-\tilde{v}_j\\[0.2em]
 \tilde{w}_j^{(4)}-\gamma \tilde{w}''_j
\end{matrix}\right)
\right\}_{j=1}^3.
\end{equation*}
\end{proposition}
\begin{proof}
For the operator $\mathcal{T}^*$ we have the identity $\<\mathcal{T}x,z> =\<x,\mathcal{T}^*z>$ for any $x\in \bm{D}\left(\mathcal{T}\right)$, $z\in \bm{D}\left(\mathcal{T}^*\right)$, and so the lemma follows by a straightforward calculation using integration by parts (and is therefore omitted).
\end{proof}

Before proving the next theorem, let us derive an intermediate result. Since $\mathbb{X}$ is a Hilbert space it follows that $\sigma\left(\mathcal{T}^*\right)=\overline{\sigma\left(\mathcal{T}\right)}$. As anticipated in \eqref{eq_14xcabaa}, we let $\lambda\in\sigma\left(\mathcal{T}\right)$ and $\mu\in\sigma\left(\mathcal{T}^*\right)$ be eigenvalues of $\mathcal{T}$ and $\mathcal{T}^*$, respectively, with corresponding eigenvectors $x$ and $z$. Then we have
\begin{equation*}
\lambda\<x, z> =\<\mathcal{T}x,z> =\<x,\mathcal{T}^*z> =\<x,\mu z> =\overline{\mu}\<x,z>, 
\end{equation*}
and since $\<x,z> =0$ for $\lambda\neq\overline{\mu}$, we only consider the case where $\lambda=\overline{\mu}$. Write out \eqref{eq_14xcabaa} to get
\begin{equation}\label{eqbevp002}
\left\{\begin{aligned}
-\tilde{v}_j&=\overline{\lambda}\tilde{w}_j,& j&=1,2,3,\\[0.2em]
\tilde{w}^{(4)}_j-\gamma \tilde{w}''_j&=\overline{\lambda}\tilde{v}_j,& j&=1,2,3\\[0.2em]
\tilde{w}_j\left(1\right)=\tilde{w}''_j\left(1\right)&=0,& j&=1,2,3,\\[0.2em]
\tilde{w}_i\left(0\right)&=\tilde{w}_j\left(0\right),& i,j&=1,2,3,\\[0.2em]
\tilde{w}''_j\left(0\right)-\alpha\tilde{w}'_j\left(0\right) +\beta \tilde{v}'_j\left(0\right)&=0,& j&=1,2,3,\\
\sum^3_{j=1}\,(\tilde{w}^{(3)}_j\left(0\right)-\gamma \tilde{w}_j'\left(0\right))&=0.&&&
\end{aligned}\right.
\end{equation}
Clearly \eqref{eqbevp002} is the linearisation of the boundary-eigenvalue problem
\begin{equation}\label{eqbevp003}
\left\{\begin{aligned}
\tilde{w}^{(4)}_j-\gamma \tilde{w}''_j&=-\overline{\lambda}^2\tilde{w}_j,& j&=1,2,3\\[0.2em]
\tilde{w}_j\left(1\right)=\tilde{w}''_j\left(1\right)&=0,& j&=1,2,3,\\[0.2em]
\tilde{w}_i\left(0\right)&=\tilde{w}_j\left(0\right),& i,j&=1,2,3,\\[0.2em]
\tilde{w}''_j\left(0\right)-(\alpha +\overline{\lambda}\beta)\, \tilde{w}'_j\left(0\right)&=0,& j&=1,2,3,\\
\sum^3_{j=1}\,(\tilde{w}^{(3)}_j\left(0\right)-\gamma \tilde{w}_j'\left(0\right))&=0,&&&
\end{aligned}\right.
\end{equation}
if we set $v_j=-\overline{\lambda}\tilde{w}_j$, $j=1,2,3$. It follows from the results of Section \ref{sec_4} that if
\begin{equation*}
x=\left\{\left(\begin{matrix}
w_j\\[0.2em]
v_j
\end{matrix}\right)\right\}^3_{j=1}=\left\{\left(\begin{matrix}
w_j\\[0.2em]
\lambda w_j
\end{matrix}\right)\right\}^3_{j=1}
\end{equation*}
is an eigenvector of $\mathcal{T}$ corresponding to an eigenvalue $\lambda$, where $w_j=w_j\left(\lambda,s\right)$, $j=1,2,3$, satisfies the boundary-eigenvalue problem \eqref{eq1sdcxabddv}, then
\begin{equation*}
z=\left\{\left(\begin{matrix}
\tilde{w}_j\\[0.2em]
\tilde{v}_j
\end{matrix}\right)\right\}^3_{j=1}=\left\{\left(\begin{matrix}
\tilde{w}_j\\[0.2em]
-\overline{\lambda} \tilde{w}_j
\end{matrix}\right)\right\}^3_{j=1}
\end{equation*}
is an eigenvector of the adjoint operator $\mathcal{T}^*$ corresponding to $\overline{\lambda}$, where $\tilde{w}_j=\tilde{w}_j\,(\overline{\lambda},s)$, $j=1,2,3$, satisfies the adjoint boundary-eigenvalue problem \eqref{eqbevp003}. We then have that $\tilde{w}_j\,(\overline{\lambda},\,\cdot\,)=\overline{{w}_j\left(\lambda,\,\cdot\,\right)}$. So the inner product of $x$ with $z$ yields
\begin{align}
\<x,z>&=\sum^3_{j=1}\left[\int^1_0w''_j\left(\lambda,s\right)\overline{\tilde{w}''_j\,(\overline{\lambda},s)}\,ds+\alpha w'_j\left(\lambda,0\right)\overline{\tilde{w}'_j\,(\overline{\lambda},0)}\right]\nonumber\\
&\qquad+\gamma \sum^3_{j=1}\int^1_0w'_j\left(\lambda,s\right)\overline{\tilde{w}'_j\,(\overline{\lambda},s)}\,ds+\sum^3_{j=1}\int^1_0v_j\left(\lambda,s\right)\overline{\tilde{v}_j\,(\overline{\lambda},s)}\,ds\nonumber\\[0.2em]
&=\sum^3_{j=1}\left[\int^1_0(w''_j\left(\lambda,s\right)\!{)}^2ds+\alpha \,(w'_j\left(\lambda,0\right)\!{)}^2\right]+\gamma \sum^3_{j=1}\int^1_0(w'_j\left(\lambda,s\right)\!{)}^2ds\nonumber\\
&\qquad-\lambda^2\sum^3_{j=1}\int^1_0 (w_j\left(\lambda,s\right)\!{)}^2ds\nonumber\\[0.2em]
&=-2\lambda^2\sum^3_{j=1}\int^1_0 (w_j\left(\lambda,s\right)\!{)}^2ds-\lambda\beta\sum^3_{j=1}\,(w'_j\left(\lambda,0\right)\!{)}^2,\label{eq_28}
\end{align}
the last line following on integrating by parts and taking into account \eqref{eq1sdcxabddv}.

We can now state and prove our main result of this subsection.
\begin{theorem}\label{T-5-3}
Each eigenvalue belonging to the first branch $\sigma^{(1)}\left(\mathcal{T}\right)$ of the spectrum of $\mathcal{T}$ has algebraic multiplicity $1$, and so is simple in the sense of Definition \ref{d03}. Each eigenvalue belonging to the second branch $\sigma^{(2)}\left(\mathcal{T}\right)$ of the spectrum is semisimple and its algebraic multiplicity is equal to $2$.
\end{theorem}
\begin{proof}
In view of what has already been referred to briefly at the beginning of the section, we show that $\<x,z> \neq 0$ ($x$ and $z$ being the eigenvectors of $\mathcal{T}$ and $\mathcal{T}^*$, respectively). We consider again the two cases as in Section \ref{sec_4} which led to Theorem \ref{T-4-1}.
\begin{enumerate}[label=,leftmargin=*,align=left,labelwidth=\parindent,labelsep=0pt,ref={\arabic*}]
\item\label{item12}\textbf{Case 1.} Let us recall first of all from Theorem \ref{T-4-1} that corresponding to each eigenvalue $\lambda\in\sigma^{(1)}\left(\mathcal{T}\right)$ there is an eigenvector
\begin{equation*}
x=\left\{\left(\begin{matrix}
w_j\\[0.2em]
\lambda w_j
\end{matrix}\right)\right\}^3_{j=1}
\end{equation*}
where $w_j\coloneqq cw$ with $w=w\left(\lambda,s\right)$ given by
\begin{equation*}
w\left(\lambda,s\right)=\mu_1\cosh\mu_3\sinh\mu_1\left(1-s\right)-\mu_3\cosh\mu_1\sinh\mu_3\left(1-s\right).
\end{equation*}
In this case \eqref{eq_28} becomes
\begin{equation}\label{eq_28xa}
\<x,z> =3c^2\left(-2\lambda^2\int^1_0 (w\left(\lambda,s\right)\!{)}^2ds
-\lambda\beta\,(w'\left(\lambda,0\right)\!{)}^2\right).
\end{equation}
We calculate for the expression in brackets (after some computations)
\begin{equation}
\begin{split}\label{eq31}
&-2\lambda^2\int^1_0 (w\left(\lambda,s\right)\!{)}^2ds-\lambda\beta\,(w'\left(\lambda,0\right)\!{)}^2\\[0.2em]
&\hspace{0.1\linewidth}=-\lambda^2\cosh\mu_1\cosh\mu_3\left(\mu_1\sinh\mu_1\cosh\mu_3+\mu_3\cosh\mu_1 \sinh\mu_3\right)\\
&\hspace{0.1\linewidth}\qquad-4\lambda^2 \frac{\mu_1\mu_3}{\mu^2_3-\mu^2_1}\cosh\mu_1\cosh\mu_3\left(\mu_1\cosh\mu_1\sinh\mu_3-\mu_3 \sinh\mu_1\cosh\mu_3\right)\\
&\hspace{0.1\linewidth}\qquad-\lambda\beta\, (\mu^2_3-\mu^2_1{)}^2\cosh^2\mu_1\cosh^2\mu_3+\lambda^2\,(\mu^2_1\cosh^2\mu_3+\mu^2_3\cosh^2\mu_1).
\end{split}
\end{equation}
Proceeding as before, by substituting in the right-hand side of \eqref{eq31} the relationships \eqref{eqrel023}, and on rearranging, we obtain
\begin{equation*}
\begin{split}
&-2\lambda^2\int^1_0 (w\left(\lambda,s\right)\!{)}^2ds
-\lambda\beta\,(w'\left(\lambda,0\right)\!{)}^2\\[0.2em]
&\hspace{0.1\linewidth}=-\rho^5\cosh\rho\cosh i\rho\left(\sinh\rho\cosh i\rho+i\cosh\rho \sinh i\rho\right)\\
&\hspace{0.1\linewidth}\qquad-2\rho^5\cosh\rho\cosh i\rho\left(\sinh\rho\cosh i\rho+i\cosh\rho\sinh i\rho+2i\rho\beta\cosh\rho\cosh i\rho\right)\\
&\hspace{0.1\linewidth}\qquad-\rho^6\,(\cosh^2 \rho+\cosh^2 i\rho).
\end{split}
\end{equation*}
Since $\Delta_1\left(\lambda\right)=0$, we have from \eqref{eq_22}, on substituting the relationships \eqref{eqrel023},
\begin{equation*}
\sinh\rho\cosh i\rho+i\cosh\rho\sinh i\rho=-\frac{2\left(\alpha+i\rho\beta\right)}{\rho}\cosh\rho\cosh i\rho
\end{equation*}
and
\begin{equation*}
\sinh\rho\cosh i\rho+i\cosh\rho\sinh i\rho+2i\rho\beta\cosh\rho\cosh i\rho=-\frac{2\alpha}{\rho}\cosh\rho\cosh i\rho.
\end{equation*}
So
\begin{equation*}
\begin{split}
&-2\lambda^2\int^1_0 (w\left(\lambda,s\right)\!{)}^2ds
-\lambda\beta\,(w'\left(\lambda,0\right)\!{)}^2\\
&\hspace{0.1\linewidth}=\left(4\alpha +2i\rho\beta\right)\rho^4\left(\cosh\rho\cosh i\rho\right)^2-\rho^6\,(\cosh^2 \rho+\cosh^2 i\rho)\neq 0
\end{split}
\end{equation*}
and there results from \eqref{eq_28xa} that $\<x,z> \neq 0$. The first assertion follows then because by Theorem \ref{T-4-1} we know that $\operatorname{dim}\operatorname{Ker}\left(\lambda I-\mathcal{T}\right)=1$ for each $\lambda\in\sigma^{(1)}\left(\mathcal{T}\right)$.
\item\label{item13}\textbf{Case 2.} For each $\lambda\in\sigma^{(2)}\left(\mathcal{T}\right)$ we know the form of the corresponding eigenvector from Theorem \ref{T-4-1}. In this case we consider $\Delta_2\left(\lambda\right)=0$, and we find by arguments analogous to those given in Case \ref{item12} that $\<x,z> \neq 0$ here also. The second assertion then follows, by Theorem \ref{T-4-1}, from the existence of two linearly independent eigenvectors associated with $\lambda\in\sigma^{(2)}\left(\mathcal{T}\right)$, $\operatorname{dim}\operatorname{Ker}\left(\lambda I-\mathcal{T}\right)=2$ for each $\lambda\in\sigma^{(2)}\left(\mathcal{T}\right)$.
\end{enumerate}
\end{proof}

\section{Completeness, minimality, and {R}iesz basisness}\label{sec_6}

Attention is now turned to the completeness and minimality properties as well as to the Riesz basisness of the eigenvectors of $\mathcal{T}$. Recall from Theorem \ref{T-5-3} that the eigenvalues of $\mathcal{T}$ are semisimple, so indeed we need only produce results for the eigenvectors.

As we have already indicated in the Introduction, the proofs use operator results from the literature, which we collect here as a convenience for the reader. The first, Lemma \ref{thm04abc}, is actually a sharper version of \cite[Theorem V.8.1]{GohbergKrein1969}, as detailed in \cite[Section V.8]{GohbergKrein1969}, while
the third is based -- in a slightly modified form -- on \cite[Theorem 1.1]{xu2005expansion}. The proof of the second result, Lemma \ref{thm04abcd}, on the minimality property may be taken, e.g., from the proof of \cite[Lemma 2.4]{MalamudEtAl2012}.
\begin{lemma}\label{thm04abc}
Let $\mathcal{K}$ be a compact skewadjoint operator on a Hilbert space $\mathbb{X}$ with $\operatorname{Ker}\mathcal{K}=\left\{0\right\}$, and let $\mathcal{S}$ be a real operator on $\mathbb{X}$ which has finite rank. Let
\begin{equation*}
\mathcal{A}=\mathcal{K}+\beta \mathcal{S}
\end{equation*}
for $\beta\geq 0$. Then the root vectors of the operator $\mathcal{A}$ are complete in $\mathbb{X}$.
\end{lemma}

\begin{lemma}\label{thm04abcd}
Let $\mathcal{A}$ be a compact operator on $\mathbb{X}$ and $\operatorname{Ker}\mathcal{A}=\left\{0\right\}$. Then the root vectors of $\mathcal{A}$ are minimal in $\mathbb{X}$.
\end{lemma}

\begin{lemma}\label{T-6-4}
Let $\mathbb{X}$ be a separable Hilbert space and let $\mathcal{A}$ be the infinitesimal generator of a $C_0$-semigroup ${U}\left(t\right)$ on $\mathbb{X}$. Suppose that the following conditions hold:
\begin{enumerate}[\normalfont(i)]
\item\label{T-6-4-a} $\sigma\left(\mathcal{A}\right)=\sigma^{(1)}\left(\mathcal{A}\right)\cup\sigma^{(2)}\left(\mathcal{A}\right)$ where the branch $\sigma^{(2)}\left(\mathcal{A}\right)=\{\lambda_k{\}}^\infty_{k=1}$, consisting entirely of normal eigenvalues;
\item\label{T-6-4-b} $\sup_{k\geq 1}\dim E\left(\lambda_k,\mathcal{A}\right)\mathbb{X}<\infty$, the $E\left(\lambda_k,\mathcal{A}\right)$ being the eigenprojections associated with the eigenvalues $\lambda_k$; and
\item\label{T-6-4-c} there is a real constant $\kappa$ such that
\begin{equation*}
\sup\,\{\operatorname{Re}{\lambda}~|~\lambda\in\sigma^{(1)}\left(\mathcal{A}\right)\}\leq \kappa\leq \inf\,\{\operatorname{Re}{\lambda}~|~\lambda\in\sigma^{(2)}\left(\mathcal{A}\right)\},
\end{equation*}
and
\begin{equation}\label{eq12vsep}
\inf_{k\neq j}\left|\lambda_k-\lambda_j\right|>0.
\end{equation}
\end{enumerate}
Then the following statements hold:
\begin{enumerate}[\normalfont(1)]
\item\label{T-6-4-d} There exist two ${U}\left(t\right)$-invariant closed subspaces $\mathbb{X}_1$ and $\mathbb{X}_2$, in the sense that ${U}\left(t\right)\mathbb{X}_1\subseteq \mathbb{X}_1$ and ${U}\left(t\right)\mathbb{X}_2\subseteq \mathbb{X}_2$, with the properties
\begin{enumerate}[\normalfont(i)]
\item $\sigma\left(\mathcal{A}|_{\mathbb{X}_1}\right)=\sigma^{(1)}\left(\mathcal{A}\right)$ and $\sigma\left(\mathcal{A}|_{\mathbb{X}_2}\right)=\sigma^{(2)}\left(\mathcal{A}\right)$; and
\item $\left\{E\left(\lambda_k,\mathcal{A}\right)\mathbb{X}_2\right\}^{\infty}_{k=1}$ forms a Riesz basis of subspaces for $\mathbb{X}_2$, and
\begin{equation*}
\mathbb{X}=\overline{{\mathbb{X}_1\oplus \mathbb{X}_2}}.
\end{equation*}
\end{enumerate}
\item\label{T-6-4-e} If $\sup_{k\geq 1}\left\|E\left(\lambda_k,\mathcal{A}\right)\right\|<\infty$, then
\begin{equation*}
\bm{D}\left(\mathcal{A}\right)\subset \mathbb{X}_1\oplus \mathbb{X}_2\subset \mathbb{X}.
\end{equation*}
\item\label{T-6-4-f} $\mathbb{X}$ can be decomposed into the topological direct sum 
\begin{equation*}
\mathbb{X}=\mathbb{X}_1\oplus \mathbb{X}_2
\end{equation*}
if and only if $\sup_{n\geq 1}\,\bigl\|\sum\limits^n_{k=1}E\left(\lambda_k,\mathcal{A}\right)\bigr\|<\infty$.
\end{enumerate}
\end{lemma}

\begin{remark}
It follows from the denseness of the domain of $\mathcal{A}$ and the compactness of its resolvent implicit in Lemma \ref{T-6-4} that, in general, the Hilbert space $\mathbb{X}$ must be separable. Further, the range $E\left(\lambda_0,\mathcal{A}\right)\mathbb{X}$ of the eigenprojection $E\left(\lambda_0,\mathcal{A}\right)$ associated with an eigenvalue $\lambda_0$ of $\mathcal{A}$ should not be confused with the geometric interpretation of $\operatorname{Ker}\left(\lambda_0 I-\mathcal{A}\right)$ as the eigenspace corresponding to $\lambda_0$. In fact from Definition \ref{d03} it is not difficult to see that the range dimension $\operatorname{dim}E\left(\lambda_0,\mathcal{A}\right)\mathbb{X}$ equals the algebraic multiplicity (and hence here, in view of Theorem \ref{T-5-3}, the geometric multiplicity of $\lambda_0$). Finally, the notion of a Riesz basis of subspaces is here equivalent to the notion of a Riesz basis with parentheses (see \cite{MR731903}), and we note that if for any eigenvalue $\lambda_0$, $\operatorname{dim}E\left(\lambda_0,\mathcal{A}\right)\mathbb{X}=1$ and the separation condition \eqref{eq12vsep} holds, then one has the usual Riesz basisness of the eigenvectors of $\mathcal{A}$.
\end{remark}

To use Lemma \ref{thm04abc}, our proof of the completeness of the eigenvectors shall be based on a particular construction of the inverse operator $\mathcal{T}^{-1}$, as shown in the following theorem.
\begin{theorem}\label{L-6-1}
Let $\mathcal{T}_0$ be the skewadjoint part of $\mathcal{T}$ corresponding to the case $\beta=0$. Then $\mathcal{T}^{-1}=\mathcal{T}^{-1}_0+\beta \mathcal{M}$, where $\mathcal{M}$ is a real operator on $\mathbb{X}$ which has finite rank.
\end{theorem}
\begin{proof}
We employ some of the results from the proof of Lemma \ref{L-3-1}, and begin by noting that on setting
\begin{align*}
b\left(\tilde{v}\right)&\coloneqq \frac{\sum^3_{j=1}\left[\tilde{V}_j\left(0\right)+\int^1_0\left(\sqrt{\gamma}\sinh\sqrt{\gamma}\,r+\alpha\cosh\sqrt{\gamma}\,r\right)\tilde{V}_j\left(r\right)dr\right]}{\gamma\sinh\sqrt{\gamma}+\alpha\sqrt{\gamma}\cosh\sqrt{\gamma}},\\[0.2em]
b\left(\tilde{w}\right)&\coloneqq \frac{\sum^3_{j=1}\tilde{w}'_j\left(0\right)}{\gamma\sinh\sqrt{\gamma}+\alpha\sqrt{\gamma}\cosh\sqrt{\gamma}},
\end{align*}
we can write \eqref{eq12a} as
\begin{equation*}
\sum^3_{j=1} a_j=b\left(\tilde{v}\right)+\beta b\left(\tilde{w}\right).
\end{equation*}
Then \eqref{eqne03} becomes
\begin{equation*}
w\left(0\right)=\frac{1}{3}\left(\left(b\left(\tilde{v}\right)+\beta b\left(\tilde{w}\right)\right)\sinh\sqrt{\gamma} -\frac{1}{\sqrt{\gamma}}\int^1_0\sinh\sqrt{\gamma}\,r\sum^3_{j=1}\tilde{V}_j\left(r\right)dr\right)= c\left(\tilde{v}\right) + \beta c\left(\tilde{w}\right),
\end{equation*}
where $c\left(\tilde{w}\right)\coloneqq \frac{b\left(\tilde{w}\right)}{3}\sinh\sqrt{\gamma}$. It is now routine to show that \eqref{eq_12a} is equivalent to
\begin{equation}\label{eq20}
\left\{\begin{split}
&a_j\sqrt{\gamma}\sinh\sqrt{\gamma}-(w^{(3)}_j\left(0\right)-\gamma  w'_j\left(0\right))\int^1_0\left(1-r\right)\sinh\sqrt{\gamma}\,r\,dr\\[0.2em]
&\hspace{0.4\linewidth}=F_{11}\left(\tilde{v}\right)+\beta F_{12}\left(\tilde{w}\right),\\[0.2em]
&a_j \,\alpha\cosh\sqrt{\gamma}-(w^{(3)}_j\left(0\right)-\gamma  w'_j\left(0\right))\,\frac{1}{\sqrt{\gamma}}\left[\alpha\int^1_0\left(1-r\right)\cosh\sqrt{\gamma}\,r\,dr +1\right]\\[0.2em]
&\hspace{0.4\linewidth}=F_{21}\left(\tilde{v}\right)+\beta F_{22}\left(\tilde{w}\right),
\end{split}\right.
\end{equation}
for $j=1,2,3$, wherein
\begin{gather*}
F_{11}\left(\tilde{v}\right)\coloneqq \int^1_0\sinh\sqrt{\gamma}\,r\tilde{V}_j\left(r\right)dr+\sqrt{\gamma}\,c\left(\tilde{v}\right),\quad
 F_{12}\left(\tilde{w}\right)\coloneqq\sqrt{\gamma}\,c\left(\tilde{w}\right),\\[0.2em]
F_{21}\left(\tilde{v}\right)\coloneqq \frac{\alpha}{\sqrt{\gamma}} \int^1_0\cosh\sqrt{\gamma}\,r\tilde{V}_j\left(r\right)dr+\frac{1}{\sqrt{\gamma}}\tilde{V}_j\left(0\right)-\sqrt{\gamma}\,c\left(\tilde{v}\right),\quad
F_{22}\left(\tilde{w}\right)\coloneqq\frac{1}{\sqrt{\gamma}}\,\tilde{w}'_j\left(0\right)-\sqrt{\gamma}\,c\left(\tilde{w}\right).
\end{gather*}
We know that \eqref{eq20} has unique solution pairs $a_j$, $w^{(3)}_j\left(0\right)-\gamma w'_j\left(0\right)$. Write these as
\begin{align*}
a_j&= \left|\begin{matrix}
 F_{11}\left(\tilde{v}\right)+\beta F_{12}\left(\tilde{w}\right)&-\int^1_0\left(1-r\right)\sinh\sqrt{\gamma}\,r\,dr\\[0.2em]
 F_{21}\left(\tilde{v}\right)+\beta F_{22}\left(\tilde{w}\right)&-\frac{1}{\sqrt{\gamma}}\left[\alpha\int^1_0\left(1-r\right)\cosh\sqrt{\gamma}\,r\,dr +1\right]\end{matrix}\right|\\
&= a_j\left(\tilde{v}\right)+\beta a_j\left(\tilde{w}\right),\quad j=1,2,3,
\end{align*}
and
\begin{align*}
w^{(3)}_j\left(0\right)-\gamma w'_j\left(0\right)&= \left|\begin{matrix}
\sqrt{\gamma}\sinh\sqrt{\gamma} & F_{11}\left(\tilde{v}\right)+\beta F_{12}\left(\tilde{w}\right)\\[0.2em]
\alpha\cosh\sqrt{\gamma} & F_{21}\left(\tilde{v}\right)+\beta F_{22}\left(\tilde{w}\right)\end{matrix}\right|\\[0.2em]
&= w^{(3)}_j\left(0,\tilde{v}\right)-\gamma w'_j\left(0,\tilde{v}\right) +\beta\,(w^{(3)}_j\left(0,\tilde{w}\right)-\gamma w'_j\left(0,\tilde{w}\right)),\quad j=1,2,3.
\end{align*}
Therefore \eqref{eq233} may be replaced by
\begin{align*}
w_j\left(s\right)&=a_j\left(\tilde{v}\right)\sinh\sqrt{\gamma}\left(1-s\right)+(w^{(3)}_j\left(0,\tilde{v}\right)-\gamma w'_j\left(0,\tilde{v}\right))\,\frac{1}{\sqrt{\gamma}}\int^1_s\left(1-r\right)\sinh\sqrt{\gamma}\left(s-r\right)dr\\
&\qquad  +\frac{1}{\sqrt{\gamma}}\int^1_s\sinh\sqrt{\gamma}\left(s-r\right)\tilde{V}_j\left(r\right)dr+\beta\,\biggl[a_j\left(\tilde{w}\right)\sinh\sqrt{\gamma}\left(1-s\right)\biggr.\\
\biggl.&\hspace{0.2\linewidth}+(w^{(3)}_j\left(0,\tilde{w}\right)-\gamma w'_j\left(0,\tilde{w}\right))\,\frac{1}{\sqrt{\gamma}}\int^1_s\left(1-r\right)\sinh\sqrt{\gamma}\left(s-r\right)dr\biggr]\\[0.25em]
&= w_j\left(s,\tilde{v}\right)+\beta w_j\left(s,\tilde{w}\right),\quad j=1,2,3,
\end{align*}
with the terms
\begin{equation*}
w_j\left(s,\tilde{w}\right)\coloneqq a_j\left(\tilde{w}\right)\sinh\sqrt{\gamma}\left(1-s\right)+(w^{(3)}_j\left(0,\tilde{w}\right)-\gamma w'_j\left(0,\tilde{w}\right))\,\frac{1}{\sqrt{\gamma}}\int^1_s\left(1-r\right)\sinh\sqrt{\gamma}\left(s-r\right)dr.
\end{equation*}
Now \eqref{eq12xvbn} can be rewritten in the form
\begin{align*}
\left\{\left(\begin{matrix}
w_j\\[0.2em]
v_j
\end{matrix}\right)\right\}^3_{j=1}=\mathcal{T}^{-1}\left\{\left(\begin{matrix}
\tilde{w}_j\\[0.2em]
\tilde{v}_j
\end{matrix}\right)\right\}^3_{j=1}&=\left\{\left(\begin{matrix}
w_j\left(\,\cdot\,,\tilde{v}\right)\\[0.2em]
v_j
\end{matrix}\right)\right\}^3_{j=1}+\beta\left\{\left(\begin{matrix}
w_j\left(\,\cdot\,,\tilde{w}\right)\\[0.2em]
0
\end{matrix}\right)\right\}^3_{j=1}\\[0.2em]
&=\mathcal{T}^{-1}_0\left\{\left(\begin{matrix}
\tilde{w}_j\\[0.2em]
\tilde{v}_j
\end{matrix}\right)\right\}^3_{j=1}+\beta \mathcal{M}\left\{\left(\begin{matrix}
\tilde{w}_j\\[0.2em]
\tilde{v}_j
\end{matrix}\right)\right\}^3_{j=1}
\end{align*}
with
\begin{align*}
\mathcal{M}\left\{\left(\begin{matrix}
\tilde{w}_j\left(s\right)\\[0.2em]
\tilde{v}_j\left(s\right)
\end{matrix}\right)\right\}^3_{j=1}&=\left\{\left(\begin{matrix}
a_j\left(\tilde{w}\right)\sinh\sqrt{\gamma}\left(1-s\right)+(w^{(3)}_j\left(0,\tilde{w}\right)-\gamma w'_j\left(0,\tilde{w}\right))\\
\times\frac{1}{\sqrt{\gamma}}\int^1_s\left(1-r\right)\sinh\sqrt{\gamma}\left(s-r\right)dr
\\[0.2em]
0
\end{matrix}\right)\right\}^3_{j=1}\\[0.2em]
&=\left(\begin{matrix}
\left\{a_j\left(\tilde{w}\right)\right\}^3_{j=1}\sinh\sqrt{\gamma}\left(1-s\right)+\{(w^{(3)}_j\left(0,\tilde{w}\right)-\gamma w'_j\left(0,\tilde{w}\right)){\}}^3_{j=1}\\
\times\frac{1}{\sqrt{\gamma}}\int^1_s\left(1-r\right)\sinh\sqrt{\gamma}\left(s-r\right)dr\\[0.2em]
0
\end{matrix}\right),
\end{align*}
where $\mathcal{T}^{-1}_0$ is a compact skewadjoint operator and $\mathcal{M}$ is a bounded linear operator of rank 2 on $\mathbb{X}$ (the $a_j\left(\tilde{w}\right)$, $w^{(3)}_j\left(0,\tilde{w}\right)-\gamma w'_j\left(0,\tilde{w}\right)$ being bounded linear functionals). Hence, in view of Lemma \ref{L-3-2}, we have, with $\mathcal{T}x=y$, $y\in\mathbb{X}$, $x\in\bm{D}\left(\mathcal{T}\right)$,
\begin{equation*}
\operatorname{Re}\<\mathcal{T}x,x> =\operatorname{Re}\<\mathcal{T}^{-1}y,y> 
=\operatorname{Re}\<(\mathcal{T}^{-1}_0+\beta \mathcal{M})\,y,y> =\beta\operatorname{Re}\<\mathcal{M}y,y>=\beta\<\mathcal{M}y,y> . 
\end{equation*}
This completes the proof.
\end{proof}

Combining Theorem \ref{L-6-1} with Lemmas \ref{thm04abc} and \ref{thm04abcd}, identifying $\mathcal{A}$ and  $\mathcal{A}_0$ with $\mathcal{T}^{-1}$ and $\mathcal{T}_0^{-1}$, respectively, and recalling from Lemma \ref{L-3-1} that $0\in\varrho\left(\mathcal{T}\right)$ and $\mathcal{T}^{-1}$ is compact, we obtain the following result guaranteeing the completeness and minimality of the eigenvectors of $\mathcal{T}$.
\begin{theorem} \label{T-6-3}
The eigenvectors of $\mathcal{T}$ are minimal complete in $\mathbb{X}$, in the sense that
\begin{equation*}
\overline{\operatorname{span}\left\{E\left(\lambda,\mathcal{T}\right)\mathbb{X}~\middle|~\lambda\in\sigma\left(\mathcal{T}\right)\right\}}=\mathbb{X}
\end{equation*}
\end{theorem}
Most essential for our purposes of proving exponential decay for the $C_0$-semigroup ${S}\left(t\right)$ is the following theorem.
\begin{theorem}\label{T-6-5}
There exists a sequence of eigenvectors of $\mathcal{T}$ which forms a Riesz basis for $\mathbb{X}$.
\end{theorem}
\begin{proof}
We identify $\mathcal{A}$ in Lemma \ref{T-6-4} with $\mathcal{T}$ and take $\sigma^{(1)}\left(\mathcal{T}\right)=\left\{-\infty\right\}$ and $\sigma^{(2)}\left(\mathcal{T}\right)=\sigma\left(\mathcal{T}\right)=\sigma_0\left(\mathcal{T}\right)$. It follows from Theorems \ref{T-5-1} and \ref{T-5-3}, together with Remark \ref{remark02w}, that the conditions \ref{T-6-4-a} to \ref{T-6-4-c} of Lemma \ref{T-6-4} are satisfied. In particular, as one would expect from Remark \ref{remark02w}, the two branches of eigenvalues are located in strips of finite width in the open left half-plane for large $ n$, and thus, due to the separation condition in \ref{T-6-4-c}, their sequences of eigenvalues are interpolating. Since $\overline{\operatorname{span}\,\{E\left(\lambda,\mathcal{T}\right)\mathbb{X}~|~\lambda\in\sigma^{(2)}\left(\mathcal{T}\right)\}}=\mathbb{X}_2$, there exists a sequence of eigenvectors of $\mathcal{T}$ which forms a Riesz basis with parentheses for $\mathbb{X}_2$. Hence, by means of Theorem \ref{T-6-3}, we can infer that the sequence of eigenvectors forms a Riesz basis with parentheses for $\mathbb{X}$ (on account of $\mathbb{X}=\mathbb{X}_2$). Combining the completeteness of the eigenvectors of $\mathcal{T}$ asserted by Theorem \ref{T-6-3} with the  interpolation property stated above, we can thus infer that the sequence of eigenvectors forms in fact a Riesz basis for $\mathbb{X}$.
\end{proof}

\section{Exponential stability}\label{sec_7}

We have collected in the previous sections all the ingredients necessary for a conclusive verification of exponential stability of the $C_0$-semigroup ${S}\left(t\right)$ or, equivalently, of the norms of solutions of \eqref{eq_09}. Theorem \ref{T-3-2} proves that $i\mathbf{R}\subset\varrho\left(\mathcal{T}\right)$ which, combined with the fact proven in Theorem \ref{T-5-1} that the eigenvalues line up along vertical asymptotes in the open left half-plane for $\beta>0$, rules out the possibility that $\operatorname{Re}\lambda_{n}\rightarrow 0$ as $n\rightarrow\infty$ for any $\lambda_n\in\sigma\left(\mathcal{T}\right)$ and shows that $\sup\left\{\operatorname{Re}{\lambda}~\middle|~\lambda\in\sigma\left(\mathcal{T}\right)\right\}<0$ and hence $\sup\left\{\operatorname{Re}{\lambda}~\middle|~\lambda\in\sigma\left(\mathcal{T}\right)\right\}\leq -\varepsilon<0$. Thus condition \eqref{eqwwr457} is satisfied. Theorem \ref{T-6-5} proves that equality holds in \eqref{eqwwr456}, whence the spectrum-determined growth assumption is satisfied. We state our main result.
\begin{theorem}\label{T-6-6}
If $\beta>0$, then the norms of the solutions to \eqref{eq_09} given by
\begin{equation*}\label{finaleq}
{x}\left(t\right)={S}\left(t\right)x_0,\quad t\ge 0,
\end{equation*}
decay exponentially to zero as $t\rightarrow \infty$ and ${S}\left(t\right)$ is, therefore, an exponentially stable $C_0$-semigroup of contractions on $\mathbb{X}$ (with
infinitesimal generator $\mathcal{T}$) satisfying \eqref{eqnewa1}.
\end{theorem}

\begin{remark}
Obviously the distribution of the spectrum of $\mathcal{T}$ in the complex plane implies that ${S}\left(t\right)$ can be extended to a $C_0$-group of bounded linear operators on $\mathbb{X}$. This is checked by use of an eigensolution expansion for ${S}\left(t\right)x_0$, which will be well defined since $\left|\operatorname{Re}\lambda_{ n}\right|<\infty$.
\end{remark}

\bigskip\noindent
\textbf{Acknowledgments.} The research described here was supported in part by the National Natural Science Foundation of China under Grant NSFC-61773277.

\bibliographystyle{plain}
\bibliography{BibLio02}

\end{document}